\newcommand{\cB}{\mathcal{B}}
\newcommand{\C}{\mathcal{C}}

\newcommand{\cE}{\mathcal{E}}
\newcommand{\cF}{\mathcal{F}}
\newcommand{\cG}{\mathcal{G}}

\newcommand{\cN}{\mathcal{N}}

\newcommand{\cP}{\mathcal{P}}

\newcommand{\cR}{\mathcal{R}}
\newcommand{\cS}{\mathcal{S}}
\newcommand{\cT}{\mathcal{T}}

\newcommand{\cX}{\mathcal{X}}
\newcommand{\cY}{\mathcal{Y}}
\newcommand{\cZ}{\mathcal{Z}}
\documentclass[sts]{imsart}

\usepackage{amsmath, amssymb, color}
\usepackage{amsthm} 
\usepackage[colorlinks,citecolor=blue,urlcolor=blue]{hyperref}
\usepackage{xcolor,colortbl}
\RequirePackage{hypernat}
\usepackage{cleveref}
\usepackage{color}
\usepackage{amsfonts, fancybox}
\usepackage{amssymb}
\usepackage[american]{babel}
\usepackage{psfrag,color}
\usepackage{dcolumn}
\usepackage[format=hang, justification=justified, singlelinecheck=false]{subcaption}
\usepackage{flafter, bm, dsfont}
\usepackage[section]{placeins}
\usepackage[final]{showkeys}

\usepackage{multirow}
\usepackage{color}
\usepackage{amsmath}
\usepackage{float}
\usepackage{mathrsfs}
\usepackage{amsfonts}
\usepackage{dsfont}
\usepackage{amssymb}
\usepackage{algorithmic, algorithm}
\usepackage{wrapfig}

\usepackage{amssymb, latexsym}

\usepackage{graphicx}
\usepackage{epstopdf}
\usepackage{tikz}
\usepackage{pgfplots}
\usepgfplotslibrary{fillbetween}

\startlocaldefs

\renewcommand{\parallel}{\,\|\,}

\newcommand*{\ep}{\varepsilon}
\newcommand*{\dist}{\mathrm}
\newcommand*{\metric}{\rho}
\DeclareMathOperator{\im}{Im}

\DeclareMathOperator{\tr}{tr}
\DeclareMathOperator{\var}{var}
\DeclareMathOperator{\supp}{psupp}
\DeclareMathOperator*{\argmax}{argmax}

\pgfplotsset{
  every axis/.append style = {thick},tick style = {thick,black},
  /tikz/normal shift/.code 2 args = {%
    \pgftransformshift{%
        \pgfpointscale{#2}{\pgfplotspointouternormalvectorofticklabelaxis{#1}}%
    }%
  },%
  range3frame/.style = {
    tick align        = outside,
    scaled ticks      = false,
    enlargelimits     = false,
    ticklabel shift   = {10pt},
    axis lines*       = left,
    line cap          = round,
    clip              = false,
    xtick style       = {normal shift={x}{10pt}},
    ytick style       = {normal shift={y}{10pt}},
    ztick style       = {normal shift={z}{10pt}},
    x axis line style = {normal shift={x}{10pt}},
    y axis line style = {normal shift={y}{10pt}},
    z axis line style = {normal shift={z}{10pt}},
  }
}

\definecolor{NYUpurple}{rgb}{0.34,0.02,0.55}

\newcommand{\eps}{\varepsilon}

\numberwithin{equation}{section}

\newtheorem{definition}{Definition}
\newtheorem{theorem}[definition]{Theorem}
\newtheorem{lemma}[definition]{Lemma}

\newtheorem{proposition}[definition]{Proposition}

\newtheorem{assumption}{Assumption}
\newtheorem{appxlem}{Lemma}[section]
\newtheorem{appxprop}[appxlem]{Proposition}
\newtheorem{appxthm}[appxlem]{Theorem}

\newcommand{\R}{{\rm I}\kern-0.18em{\rm R}}
\newcommand{\h}{{\rm I}\kern-0.18em{\rm H}}
\newcommand{\K}{{\rm I}\kern-0.18em{\rm K}}
\newcommand{\p}{{\rm I}\kern-0.18em{\rm P}}
\newcommand{\E}{{\rm I}\kern-0.18em{\rm E}}

\newcommand{\1}{{\rm 1}\kern-0.24em{\rm I}}
\newcommand{\N}{{\rm I}\kern-0.18em{\rm N}}

\newcommand{\bone}{\mathbf{1}}

\definecolor{MIT}{RGB}{163,31,52}
\endlocaldefs

\newcommand{\sg}{\mathsf{g}}
\newcommand{\ZZ}{\mathbb{Z}}
\newcommand{\RR}{\mathbb{R}}

\crefname{theorem}{Theorem}{Theorems}
\crefname{observation}{Observation}{Observations}
\crefname{claim}{Claim}{Claims}
\crefname{condition}{Condition}{Conditions}
\crefname{example}{Example}{Examples}
\crefname{fact}{Fact}{Facts}
\crefname{lemma}{Lemma}{Lemmas}
\crefname{corollary}{Corollary}{Corollaries}
\crefname{definition}{Definition}{Definitions}
\crefname{remark}{Remark}{Remarks}

\makeatletter
\newcommand*{\defeq}{\mathrel{\rlap{%
                     \raisebox{0.3ex}{$\m@th\cdot$}}%
                     \raisebox{-0.3ex}{$\m@th\cdot$}}%
                    =}
\newcommand*{\eqdef}{=
  \mathrel{\rlap{%
      \raisebox{0.3ex}{$\m@th\cdot$}}%
    \raisebox{-0.3ex}{$\m@th\cdot$}}%
}
\makeatother
\usepackage{comment}

\begin{document}

\begin{frontmatter}

\title{Optimal rates of estimation for multi-reference alignment}
\runtitle{Optimal rates for MRA}

\begin{aug}

\author{\fnms{Afonso S.}~\snm{Bandeira}\thanksref{t1}\ead[label=afonso]{bandeira@cims.nyu.edu}},
\author{\fnms{Philippe}~\snm{Rigollet}\thanksref{t2}\ead[label=rigollet]{rigollet@math.mit.edu}},
\and
\author{\fnms{Jonathan}~\snm{Weed}\thanksref{t3}\ead[label=jon]{jweed@mit.edu}}

\affiliation{Courant Institute of Mathematical Sciences, New York University\\ Massachusetts Institute of Technology\\Massachusetts Institute of Technology}

\thankstext{t1}{Part of this work was done while A. S. Bandeira was with the Mathematics Department at MIT and supported by NSF Grant DMS-1317308.}
\thankstext{t2}{This work was supported in part by NSF CAREER DMS-1541099, NSF DMS-1541100, DARPA W911NF-16-1-0551, ONR N00014-17-1-2147 and a grant from the MIT NEC Corporation.}
\thankstext{t3}{This work was supported in part by NSF Graduate Research Fellowship DGE-1122374.}

\address{{Afonso S. Bandeira}\\
{Department of Mathematics} \\
{Courant Institute of Mathematical Sciences}\\
{Center for Data Science} \\
{New York Univeristy,}\\
{New York, NY 10012, USA}\\
\printead{afonso}
}

\address{{Philippe Rigollet}\\
{Department of Mathematics} \\
{Massachusetts Institute of Technology}\\
{77 Massachusetts Avenue,}\\
{Cambridge, MA 02139-4307, USA}\\
\printead{rigollet}
}

\address{{Jonathan Weed}\\
{Department of Mathematics} \\
{Massachusetts Institute of Technology}\\
{77 Massachusetts Avenue,}\\
{Cambridge, MA 02139-4307, USA}\\
\printead{jon}
}

\runauthor{Bandeira, Rigollet and Weed}
\end{aug}

\begin{abstract}In this paper, we establish optimal rates of adaptive estimation of a vector in the multi-reference alignment model, a problem with important applications in fields such as signal processing, image processing, and computer vision, among others. We describe how this model can be viewed as a multivariate Gaussian mixture model under the constraint that the centers belong to the orbit of a group. This enables us to derive matching upper and lower bounds that feature an interesting dependence on the signal-to-noise ratio of the model. Both upper and lower bounds are articulated around a tight local control of Kullback-Leibler divergences that showcases the central role of moment tensors in this problem.
\end{abstract}


\end{frontmatter}
\section{Introduction}
A fundamental problem arising in various scientific and engineering domains is the presence of heterogenous data. In many applications, each observation of an object of interest is corrupted not only by noise but also by a latent transformation, which can often be modeled as the action of an unknown element of a known group. The presence of these latent transformations raises serious challenges, both in theory and in practice.

Our goal in this work is to inaugurate the statistical study of such models and establish optimal rates of estimation for a particular version known in the computer science literature as \emph{multi-reference alignment}, a simple problem arising in fields such as structural biology~\cite{SchValNun05,TheSte12,Sad89}, image recognition~\cite{Bro92}, and signal processing~\cite{ZwaHeiGel03}.
The tools we develop to prove these bounds provide a unified theoretical framework for statistical estimation in the presence of algebraic structure.

\subsection{Algebraically structured models and cryo-EM}
A primary motivation to study models with algebraic structure is cryo-electron microscopy (cryo-EM), an important technique to determine three-dimensional structures of biological macromolecules. The citation for the 2017 Nobel prize in Chemistry, awarded to its inventors, reads:

\begin{quote}
The Nobel Prize in Chemistry 2017 was awarded to Jacques Dubochet, Joachim Frank and Richard Henderson ``for developing cryo-electron microscopy for the high-resolution structure determination of biomolecules in solution".
\end{quote}

In this imaging technique, each measurement consists of a noisy tomographic projection of a rotated---by an unknown rotation in $\mathrm{SO}(3)$---copy of an unknown molecule. The task is then to reconstruct the molecule density from many such measurements. This reconstruction problem has received significant attention, primarily from computational perspectives, but its statistical properties remain largely unexplored. This problem features three singular characteristics: (i) The latent group action in each observation---here a rotation---(ii) the tomographic projection and (iii) the presence of high noise as illustrated by Figure~\ref{FIG:noisy_image}.

\begin{wrapfigure}[16]{r}{0.4\textwidth}
\vspace{-1.25em} 
\begin{center}
\includegraphics[width = 0.38\textwidth]{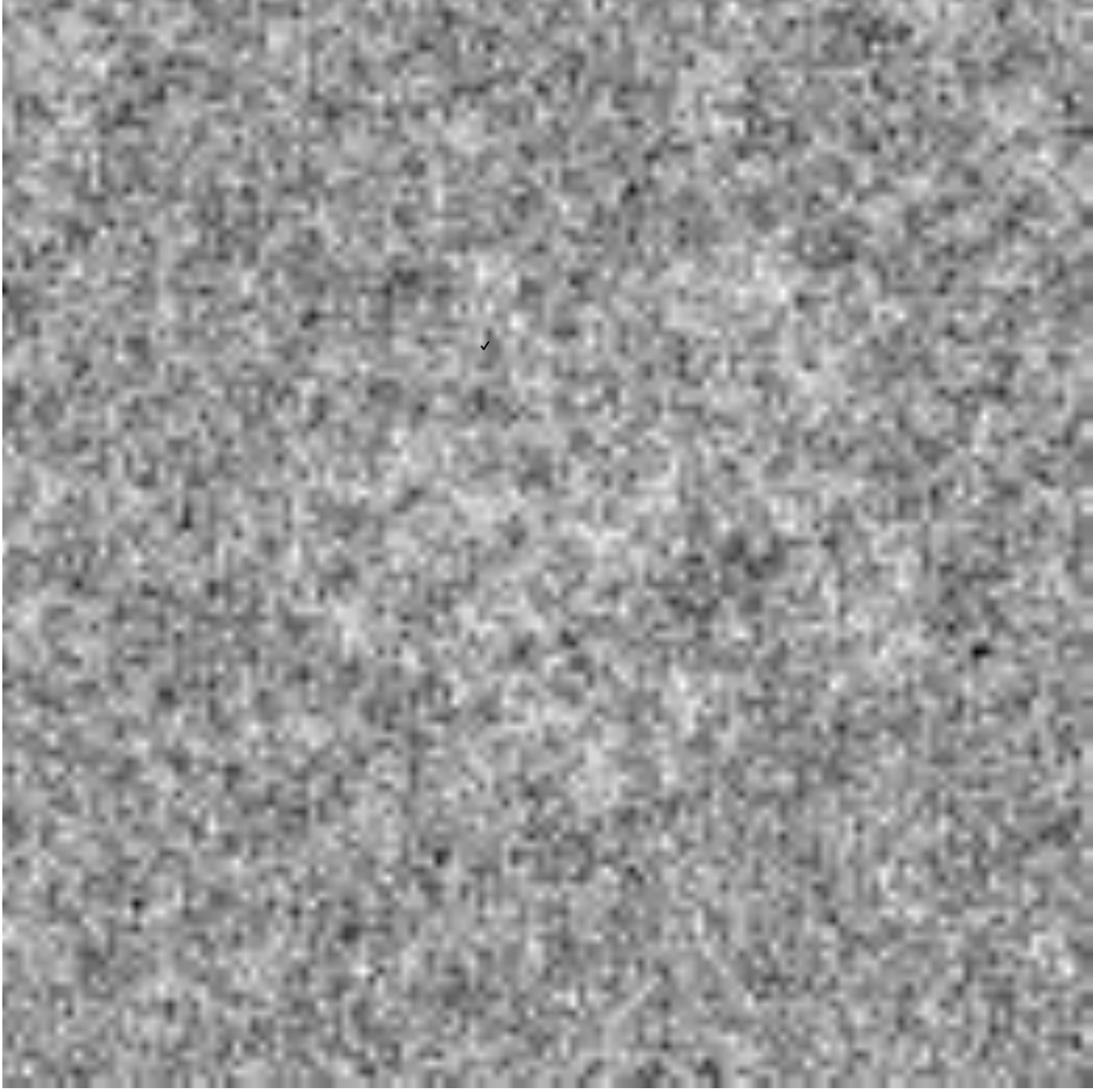}
\caption{Sample image from the E. coli 50S ribosomal subunit [Source: Sigworth Lab (Yale Medical School)].}
\label{FIG:noisy_image}
\end{center}
   \vspace{1.51em} 
  \end{wrapfigure} 

As a first step toward the statistical analysis of this class of algebraically structured models, we focus on a simpler model that features two of the aforementioned characteristics, namely (i) the presence of a group action and (iii) the presence of high noise. This model is simpler to analyze and already presents fundamentally novel statistical features that manifest themselves in nonclassical rates of estimation.

Denote by $\cG$ a known compact subgroup of the group $\mathrm{O}(L)$ of orthogonal transformations of $\R^L$. Throughout this paper, we identify the action of a group element $G \in \cG \subset \mathrm{O}(L)$ on $\R^L$ by left-multiplication with an orthogonal matrix $G$. We slightly abuse terminology by referring to $G$ as a group element. Our goal is to recover a parameter $\theta \in \R^L$, which we often refer to as a \emph{signal}, on the basis of very noisy observations corrupted by unknown elements of $\cG$. Concretely, we observe
\begin{equation}\label{eqn:basic_model}
Y_i = G_i \theta + \sigma \xi_i\,,
\end{equation}
where $G_i \in \cG$ is unknown and $\xi_i$ is standard Gaussian noise independent of $G_i$. The parameter $\theta$ is only identifiable up to the action of $\cG$, so we focus on obtaining an estimator $\tilde \theta$ whose distance to the orbit of $\theta$ as defined by
\begin{equation*}
\rho(\tilde \theta, \theta) := \min_{G \in \cG} \|\tilde \theta - G \theta\|
\end{equation*}
is small in expectation.
We call~\eqref{eqn:basic_model} an \emph{algebraically structured model}. For normalization purposes, we assume that $c^{-1} \leq \|\theta\| \leq c$ for some universal positive constant $c$. Fixing the scaling of $\theta$ in such a way allows us to control the signal-to-noise ratio of the problem only via the parameter $\sigma$, which plays a central role in the sequel.

\subsection{Prior work: The synchronization approach}\label{sec:synchro}
The difficulty of algebraically structured models resides in the fact that both the signal $\theta$ and the transformations $G_1,\ldots,G_n \in \cG$ are unknown and the latter are therefore latent variables. If the group elements were known, one could easily estimate the vector $\theta$ by taking the average of $G_i^{-1} Y_i, i =1, \ldots, n$. In fact, this simple observation is the basis of the leading current approach to this problem, called the ``synchronization approach"~\cite{BanChaSin14,BanCheSin15}. Specifically, synchronization aims at recovering the latent variables $G_{i}$ by solving a 
problem of the form
\begin{equation}\label{eqn:latent_rotations}
\min_{H_{1}, \dots, H_{n} \in \cG} \sum_{1 \leq i, j \leq n} \big\|{H}_{i}^{-1} Y_i - {H}_{j}^{-1} Y_j\big\|^2\,.
\end{equation}
Denoting by $\tilde H_{i}$ the solutions of~\eqref{eqn:latent_rotations}, one can then estimate $\theta$ by the average of ${\tilde H}^{-1}_i Y_i, i =1, \ldots, n$. 
Despite synchronization problems being computationally hard in general~\cite{BanChaSin14}, certain theoretical guarantees have been derived under specific noise models that are unfortunately not realistic for the problems of interest in this paper. For example, it is often assumed that each \emph{pair} of observations is corrupted by independent noise, so that the terms in the sum  in~\eqref{eqn:latent_rotations} are independent. Instead, our model adopts the more relevant assumption of independent noise on each observation.
Among the most prominent methods to date are spectral methods~\cite{Sin11,BanSinSpi11}, semidefinite relaxations~\cite{BanChaSin14,BanCheSin15,AbbBanBra14b,BanBouSin16,JavMonRic16,BanBouVor16}, methods based on Approximate Message Passing~\cite{PerWeiBan16a} and other modified power methods~\cite{Boumal_ProjPowerMethod,Chen_Candes_ProjPowerMethod}. Synchronization also enjoys many interesting connections with geometry (see, e.g., \cite{Gao_GeometrySynchronization}).

Another fundamental drawback of the synchronization approach is its intolerance to large noise levels $\sigma$.
When $\sigma$ is significantly smaller than $\|\theta\|$, the prior work referenced above has demonstrated empirically and theoretically that the synchronization approach yields excellent results.
Intuitively, the success of this approach relies on the fact that when the noise is small, macroscopic features of the underlying signal are still visible.
However, as noted in our discussion of cryo-EM, the noise level in applications is often significantly larger than the signal~\cite{Sig16}, which renders the synchronization approach unusable. An illustration of the difference between these regimes appears in Figure~\ref{fig:noise_low_high}.

\begin{figure}[h]
\begin{center}
\includegraphics[width=\textwidth]{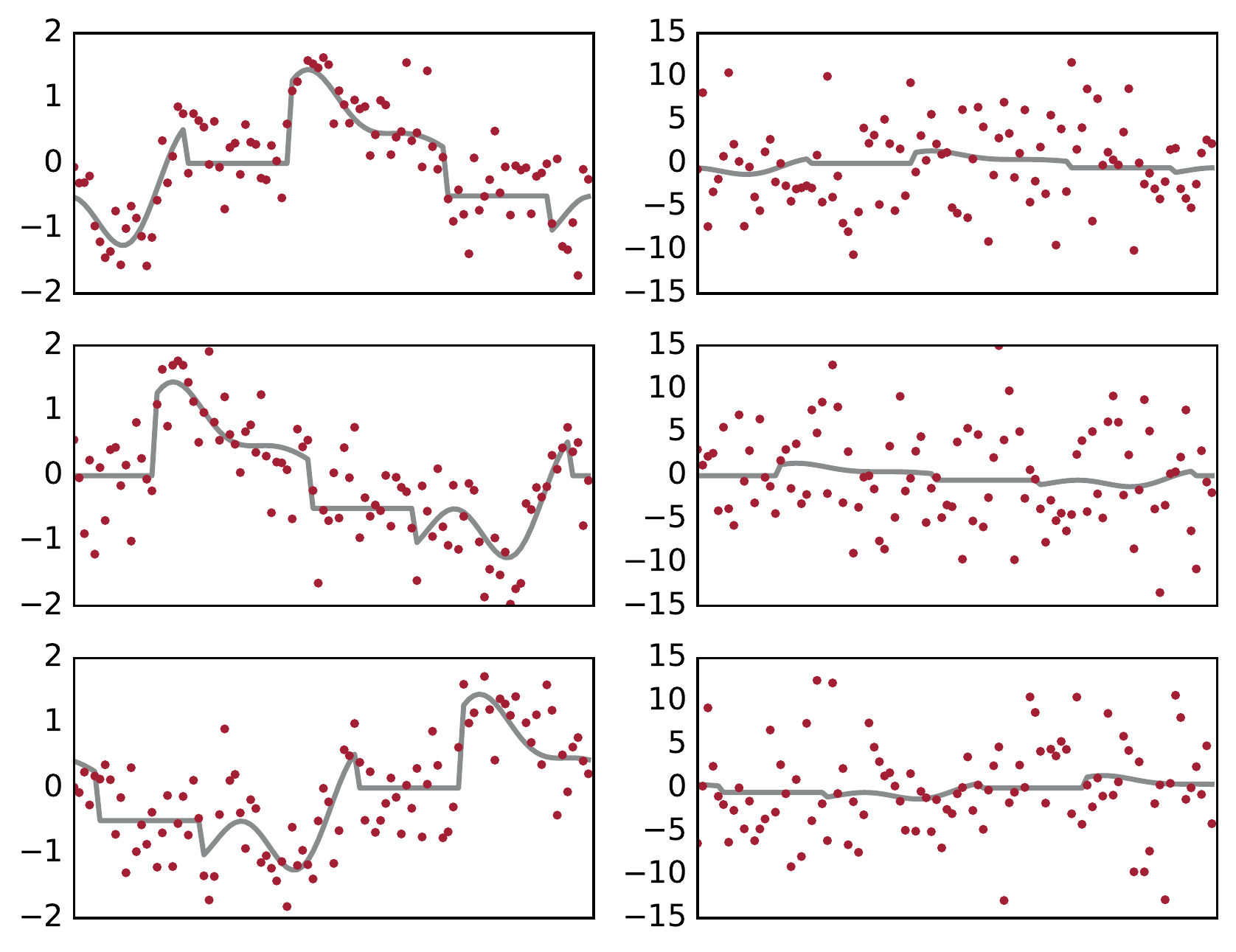}
\caption{Instances of the multi-reference alignment problem, defined in Section~\ref{sec:mra}, at low (left column) and high (right column) noise levels. The true underlying signal appears in gray, and the noised version appears in red. When the noise level is low, large features of the signal are still visible despite the noise; in the presence of large noise, however, the signals cannot reliably be synchronized.}
\label{fig:noise_low_high}
\end{center}
\end{figure}

From a theoretical standpoint, this fact implies that the low- and high-noise regimes are very different: when the noise is sufficiently large, prior work has shown that the transformations are impossible to reliably estimate, regardless of the number of samples~\cite{Weinstein_limitssignalalignment,Sapiro_limitsimagealignment}.
Thus, for the high-noise regime, new techniques are required. We therefore focus in this work on the case where the variance of the noise is bounded below by a constant.

\subsection{The Gaussian mixture approach}\label{SEC:setup}
We propose an alternative to the synchronization approach that completely bypasses the estimation of the transformations $G_1,\dots,G_n$ in favor of estimating $\theta$ directly. To do so, we first show how to recast our model as a continuous mixture of Gaussians whose centers are algebraically constrained.

To reinterpret~\eqref{eqn:basic_model} as a Gaussian mixture model, we replace the latent group elements $G_1, \dots, G_n$ by group elements drawn independently and uniformly at random (according to the Haar measure) from $\cG$.
This is a worst-case assumption, which is appropriate since we prove minimax rates.\footnote{Following an earlier version of this paper, \cite{AbbBenLee17} considered a version of multi-reference alignment when the distribution of $G$ is not uniform and showed that strictly better rates can be obtained in some cases.} Indeed, we can always reduce to this case: since the Gaussian distribution is invariant under the action of the orthogonal group, we can transform each observation $Y_i$ into $H_i Y_i$, where $H_i$ is uniformly distributed over $\cG$ and independent of all other random variables. Since $H_i G_i$ is also uniformly distributed over $\cG$, these new observations are drawn from a mixture of Gaussians whose centers are given by $G\theta$, $G \in \cG$, with uniform mixing weights. In particular, these centers are linked together by a rigid algebraic structure: they are the orbit of $\theta \in \R^L$ under the action of $\cG$.

We therefore specify the following Gaussian mixture model. Given a noise level $\sigma$, group $\cG \subseteq \mathrm{O}(L)$, and parameter of interest $\theta \in \R^L$, denote by $\dist{P}_\theta$ the distribution of a random variable $Y$ satisfying
\begin{equation}\label{eq:MRA:model}
Y = G\theta + \sigma \xi\,,
\end{equation}
where $G$ is drawn uniformly from $\cG$ and $\xi \sim \cN(0, I_L)$ is independent Gaussian noise.

We assume throughout that the noise variance $\sigma^2$ is known. This assumption is realistic in many applications such as imaging or signal processing, where it is inexpensive to collect pure-noise samples from $\dist{P}_0$ and thereby estimate $\sigma^2$ to arbitrary accuracy.

In this work, we analyze the maximum likelihood estimator (MLE) $\tilde \theta_n$ for~\eqref{eq:MRA:model}:
\begin{equation}\label{eq:ML:mixtureGaussian}
\tilde \theta_n=\argmax_{\phi \in \R^L} \sum_{i=1}^n \log \E\big[ \exp\big(- \frac{1}{2\sigma^2} \|Y_i - G \phi\|^2\big)\big|Y_i\big]\,.
\end{equation}
We focus on obtaining the optimal scaling of the quantity $\rho(\tilde \theta_n, \theta)$ with the signal-to-noise ratio of the problem.\footnote{Our focus in this work is on statistical properties rather than on computation. In a companion paper~\cite{PerWeeBan17}, we propose and analyze a computationally efficient estimator for multi-reference alignment.}  This question is central to signal processing problems where $\sigma$ is quite large, since it determines the order of magnitude of the sample size $n$ required to achieve a certain accuracy.
Moreover, in many applications, technological improvements to the measurement apparatus can directly improve the effective value of $\sigma$---in cryo-EM, for instance, this is a focus of active research~\cite{Sig16}.
For these reasons, understanding the scaling of $\rho(\tilde \theta_n, \theta)$ with $\sigma$ is a core question both in theory and in practice.
Our main upper bound result gives a \emph{uniform} analysis of this maximum likelihood estimator, valid for any algebraically constrained model. We complement this analysis with lower bounds which are equally universal. In both cases, we proceed by controlling the Fisher information of the model.

Gaussian mixture models have been extensively studied in the statistical literature since their introduction by~\cite{Pea94} in the nineteenth century (see, e.g.,~\cite{McLPee00} for an overview). As illustrated by the extant literature, mixture models are quite rich and broadly applicable to a variety of statistical problems ranging from clustering to density estimation. It is known that the rate of estimation of the parameters of a Gaussian mixture with $k$ components can scale like $n^{-O(1/k)}$ (see for example \cite{Che95,MoiVal10} and more recently~\cite{HeiKah15} %
for an interesting explanation from the point of view of model misspecification). In this work, our analysis of the multi-reference alignment problem focuses on a setting where the convergence of $\tilde \theta_n$ to $\theta$ occurs at the parametric $n^{-1/2}$ rate; nevertheless, our results show that even in this benign setting, the optimal dependence of this rate on $\sigma$ can still be extremely poor.

\subsection{Multi-reference alignment}\label{sec:mra}
As an application of our techniques, we analyze and establish optimal rates for a model known as multi-reference alignment, a simple algebraically structured model.
Multi-reference alignment is a special case of cryo-EM, where instead of three-dimensional rotations we consider phase shifts of a periodic signal.
This represents a special case of cryo-EM because it corresponds to the situation when the axis of rotation of the molecule is known, but not its angle.

In addition to being a toy model for cryo-EM, this simpler model is also of independent interest in several applications including in structural biology~\cite{TheSte12} and radar classification~\cite{ZwaHeiGel03}. A discrete version of this problem where the group is the cyclic group $\ZZ/L$ acting on the coordinates of $\theta$ was introduced in~\cite{BanChaSin14} to permit approaches based on semi-definite programming; however, our results indicate that this simplification is not actually benign, in the sense that the discretized model admits significantly worse rates of estimation than the model we describe below. We compare our more general model with theirs in Appendix~\ref{supp}.

Let $f: [0, 1] \to \R$ be an unknown function, and let $g_{s}$ be the shift operator which acts on $f$ by $g_{s}\hspace{-.5ex}\circ\hspace{-.5ex} f(x) = f(x + s)$, where $s \in [0, 1)$ and the addition is performed modulo~$1$.
These operators clearly form a group, denoted $\cS$, which is isomorphic to $\R/\ZZ$.
We observe independent copies of
\begin{equation}\label{eq:continuous_model}
Y = g_{s} \hspace{-.5ex}\circ\hspace{-.5ex} f(x) + \sigma \xi\,,
\end{equation}
where where $g_{s} \hspace{-.5ex}\circ\hspace{-.5ex} f(x) \in \R^L$ denotes the vector $(g_s\hspace{-.5ex}\circ\hspace{-.5ex} f(x_j))_{j=1}^L$ for the fixed design $x_j = j/L$, $s$ is drawn uniformly at random from $[0, 1]$, and $\xi \sim \cN(0, I_L)$ is independent of $s$.

To put~\eqref{eq:continuous_model} into the same form as~\eqref{eq:MRA:model}, assume that the function $f$ is \emph{band limited}---i.e., the Fourier transform of $f$ vanishes outside the interval $[-B, B]$ for some positive integer $B$---and that the measurements are performed above the \emph{Nyquist frequency}---i.e., $L > 2B$.
This assumption ensures that the function $f$ is identifiable, in that the discrete measurements $f(x_1), \dots, f(x_L)$ suffice to recover $f$.

The action of $\cS$ on $f(x)$ can be identified with a subgroup of the orthogonal group $\mathrm{O}(L)$ by passing to the Fourier domain.
Indeed, since $f$ is band limited, it can be identified with the vector of its $2B + 1$ Fourier coefficients $(\hat f(-B), \dots, \hat f(B))$, which we denote by $\hat f$.
Writing
\begin{equation*}
f(x) = \sum_{k=-B}^{B} \hat f(k) e^{-2\pi \mathrm{i} k x}\,,
\end{equation*}
yields the relation
\begin{equation*}%
\widehat{g_{s} \hspace{-.5ex}\circ\hspace{-.5ex} f}(k) = \hat f(k) e^{-2 \pi \mathrm{i} s k} =: \hat f(k) z^k\,,
\end{equation*}
where $z = e^{-2 \pi \mathrm{i} s}$ is a complex number of unit norm.
This identifies $\cS$ with the circle group $U(1)$.%

Writing $\theta$ for the vector $f(x)$, we obtain an example of~\eqref{eq:MRA:model}: we observe independent copies of
\begin{equation}\label{eq:phase_shift}
Y = G_{z} \theta + \sigma \xi\,,
\end{equation}
where $\theta \in \R^L$ is the parameter to be estimated, $z$ is drawn uniformly at random from $U(1)$, $\xi$ is a standard Gaussian random variable independent of $z$, and
$G_{z} \theta$ is defined by its Fourier coefficients:
\begin{equation}\label{eq:discrete_action}
\widehat{G_z \theta}_k = z^k \hat \theta_k \text{ for $k = -\lfloor L/2 \rfloor, \dots, \lfloor L/2 \rfloor$,}
\end{equation}
where we use the notation $\hat \theta$ to represent the discrete Fourier transform of $\theta$.
If we restrict $z$ to be of the form $\omega^L$, where $\omega$ is a primitive $L$th root of unity, then we recover the discrete model of~\cite{BanChaSin14}. We call model~\eqref{eq:phase_shift}, the \emph{phase shift} model.

\subsection{Organization of the paper}
In Section~\ref{section_mainresults} we present our main results, Theorems~\ref{thm:main_lower} and~\ref{thm:main_upper}, providing minimax rates for the multi-reference alignment problem under the phase shift model~\eqref{eq:phase_shift}. 
The proofs of these theorems rely on developing general tools for analyzing algebraically structured models and controlling the Kullback-Leibler divergence between distributions corresponding to two different signals.

In Section~\ref{sec:uniform_upper}, we give guarantees on the maximum likelihood estimator (MLE) under a condition on the curvature of the KL divergence.
We then specialize to the phase shift model in Section~\ref{SEC:modMLE} and develop a modified MLE for the phase shift model which achieves the optimal rates in Theorem~\ref{thm:main_upper}.
Section~\ref{sec:multi} concludes by establishing the lower bound in Theorem~\ref{thm:main_lower}; the proof involves finding pairs of different signals with several matching invariant moment tensors.
Both lower and upper bounds depend on an analysis of the KL divergence for algebraically structure models, which appears in Appendix~\ref{sec:tight_bounds}.

\subsection{Notation}
We define the Fourier transform $\hat \theta$ of $\theta \in \R^L$ by
\begin{equation*}
\hat \theta_j = \frac{1}{\sqrt L} \sum_{k=1}^L e^{2 \pi \mathrm{i} j k /L} \theta_k\,, \quad \quad - \lfloor L/2 \rfloor \leq j \leq \lfloor L/2 \rfloor\,.
\end{equation*}
We assume for convenience throughout that $L$ is odd.

The symbol $\|\cdot \|$ denotes the $\ell_2$ norm on $\R^L$. For any positive integer $k$, we write $[k]=\{1, \ldots, k\}$. We use $z^*$ to denote the complex conjugate of $z \in \mathbb{C}$

Given a vector $t$, let $t^{\otimes k}$ denote the order-$k$ tensor formed by taking the $k$-fold tensor product of $t$ with itself.
Denote by $\|A\|$ the Hilbert-Schmidt norm of a tensor $A$, defined by $\|A\|^2 = \langle A, A\rangle$, where $\langle \cdot, \cdot \rangle$ denotes the entrywise inner product. It is easy to check that, for any two column vectors $t, u$ of the same size, the identity $\langle t^{\otimes k}, u^{\otimes k}\rangle=(t^\top u)^k$ holds. 

A tensor $A$ is \emph{symmetric} if 
$A_{i_1\dots i_k} = A_{i_{\pi(1)} \dots i_{\pi(k)}}$
for any permutation $\pi$ of $[k]$. For such tensors, the value $A_{i_1\dots i_k}$ depends only on the multiset $\{i_1,\dots,i_k\}$, or equivalently on the multi-index $\alpha$ defined by $\alpha_\ell = |\{j \in [k]: i_j = \ell\}|$.

The Kullback-Leibler (KL) divergence between two distributions $P$ and $Q$ such that $P \ll Q$ is given by
\begin{equation*}
D(P \parallel Q) = \int\log \big(\frac{dP}{dQ}\big) \,dP\,.
\end{equation*}
It is well known that $D(P \parallel Q) \geq 0$, with equality holding iff $P = Q$.

\section{Main results}
\label{section_mainresults}

As mentioned above, the rescaled loss $\sqrt{n}\rho(\tilde \theta_n, \theta)$ of the maximum depends asymptotically on the Fisher information of the model, which can be related to the curvature of the Kullback-Leibler divergence around its minimum. Conversely, (lack of) curvature of the Kullback-Leibler divergence around its minimum is what controls minimax lower bounds that are valid for \emph{any} esitmator.
We provide a unified framework for proving upper and lower bounds based on the curvature of the divergence function, following an idea originally introduced in \cite{LepNemSpo99} in the context of functional estimation and further developed by \cite{JudNem02,CaiLow11,WuYan16,CarVer17,BlaCarGut17}. In the multi-reference alignment model, this approach allows us to relate Kullback-Leibler divergence to moment tensors, which can in turn be controlled using Fourier-theoretic arguments.

Our analysis establishes that the difficulty of estimating a particular signal $\theta$ depends on the support of the Fourier transform $\hat \theta$ of $\theta$.
Define the \emph{positive support} $\supp(\hat \theta)$ of $\hat \theta$ by
\begin{equation*}
\supp(\hat \theta) = \{j \mid j \in \{1, \dots,  L/2 \}, \hat \theta_j \neq 0\}\,.
\end{equation*}
We focus only on the positive indices because the signal $\theta$ is real, so the Fourier transform is conjugate symmetric: $\hat \theta_j^* = \hat \theta_{-j}$.

We make the following assumptions.%
\begin{assumption}\label{assumption:norm}
There exists an absolute constant $c > 1$ such that $c^{-1} \leq \|\theta\| \leq c$.
\end{assumption}
\begin{assumption}\label{assumption}
Moreover, there exists an absolute constant $c_0$, not depending on $n$, such that $c_0 \leq |\hat \theta_j|$ for all $j \in \supp(\hat \theta)$.
\end{assumption}

We denote by $\cT$ the set of vectors satisfying Assumptions~\ref{assumption:norm} and~\ref{assumption}.
Assumption~\ref{assumption:norm} is benign and is adopted for normalization purposes, so that $\sigma$ captures entirely the signal-to-noise ratio of the problem.
Regarding Assumption~\ref{assumption}, we emphasize that this is the situation of most interest to practitioners: the existence of very small, but non-zero, coordinates whose values approach $0$ with $n$ should rightly be considered pathological. Assumption~\ref{assumption} rules out certain artificial situations analogous to classical difficulties arising in estimating mixtures of Gaussians, such as distinguishing the mixture $.5\cN(+\eps, 1) + .5\cN(-\ep, 1)$ from the single Gaussian $\cN(0, 1)$ for very small $\eps$.
Determining minimax rates of estimation without Assumption~\ref{assumption} is certainly of theoretical interest, and we leave this question for future work.

As noted above, our results focus on understanding how minimax rates of estimation for the multi-reference alignment problem scale with $\sigma$.
This is the question of primary interest in algebraically structured problems like cryo-EM, since in these applications $\sigma$ is the only part of the model that can be improved by the development of new imaging technologies and techniques.
We note that our results do not address the dependence on the dimension $L$, and obtaining sharp dependence on $L$ is an attractive open problem.

The following theorem reveals a surprising phenomenon: even under Assumption~\ref{assumption}, the multi-reference alignment problem suffers from the curse of dimensionality.
We prove the following lower bound for the phase shift model.
\begin{theorem}\label{thm:main_lower}
Let $0 \leq s \leq \lfloor L/2 \rfloor$.
Let $\cT_s$ be the set of vectors $\theta \in \cT$ satisfying %
$\supp(\hat \theta) \subset [s]$.
For any $\sigma \geq \max_{\theta \in \cT_s} \|\theta\|$, the phase shift model satisfies
\begin{equation}\label{eqn:minimax_lower}
\inf_{T_n} \sup_{\theta \in \cT_s} \E_\theta[\metric(T_n, \theta)] \geq C\Big( \frac{\sigma^{(2s-1) \vee (s+1)}}{\sqrt n} \wedge 1\Big)\,,
\end{equation}
where the infimum is taken over all estimators $T_n$ of $\theta$ and where $C$ is a universal constant.
\end{theorem}

The $\sigma^{s+1}/\sqrt n$ rate in Theorem~\ref{thm:main_upper} holds only in the edge case when $s \in \{0, 1\}$; for $2 \leq s \leq \lfloor L/2 \rfloor$ the rate scales as $\sigma^{2s-1}/\sqrt n$.

In Section~\ref{sec:upper_proof}, we show that a modified version of the MLE achieves the optimal rate asymptotically for $2 \leq s \leq \lfloor L/2 \rfloor$. This estimator is also adaptive to the class $\cT_s$.
\begin{theorem}\label{thm:main_upper}
For any $\sigma \geq 1$ and $2 \leq s \leq \lfloor L/2 \rfloor$, the modified MLE $\check \theta_n$ for the phase shift model satisfies
\begin{equation}\label{eqn:minimax_upper}
\sup_{\theta \in \cT_s} \E_\theta[\metric(\check \theta_n, \theta)] \leq C' \frac{\sigma^{2s -1}}{\sqrt n} + C_\sigma \frac{\log n}{n}\,,
\end{equation}
where $C_\sigma \leq C'' \sigma^{12 s - 11}$ and $C'$ and $C''$ are constants depending on $L$ and $c_0$ but on no other parameter.
\end{theorem}

Theorem~\ref{thm:main_upper} excludes the cases where $s = 0$ or $s = 1$.
The behavior of these cases is different, and is significantly easier to analyze.
\begin{theorem}\label{thm:small_s}
If $s \in 
\{0,1\}$ and $\sigma \geq 1$, then the phase shift model satisfies
\begin{equation*}
\inf_{T_n} \sup_{\theta \in \cT_s} \E_\theta[\metric(T_n, \theta)] \leq C''' \frac{\sigma^{s+1}}{\sqrt n}\,,
\end{equation*}
where $C'''$ is a constant depending on $c_0$ but on no other parameter.
\end{theorem}
Theorem~\ref{thm:small_s} is proved in Appendix~\ref{supp}, where we exhibit a computationally efficient estimator achieving the upper bounds for $s \in \{0, 1\}$.

A few remarks are in order. We have given rates over the classes $\cT_s$ because, in the context of cryo-EM, it is generally assumed that band-limited signals, that is, signals lying in $\cT_s$ for $s$ small, are easier to estimate. Our work offers partial validation for this view. However, we stress that the dependence on $s$ present in Theorems~\ref{thm:main_lower} and~\ref{thm:main_upper} is a consequence of the minimax paradigm. Indeed, our proof of the lower bound involves a class of signals with very specific support in the Fourier domain. Such signals drive the worst case bound of order $\sigma^{2s-1}/\sqrt{n}$.
This is in striking contrast to the behavior for signals which are likely to arise in practice---in a companion paper~\cite{PerWeeBan17}, we show that signals whose Fourier transform has full support can be estimated at the rate $\sigma^3/\sqrt n$.

Second, our proof techniques do not allow us to remove the $\sigma$ dependence of the term  $C_\sigma \log n/n$ in the upper bound. In particular, for small values of $n$, this term may actually dominate. We conjecture that this issue is an artifact of our proof technique and note that preliminary numerical results in~\cite{PerWeeBan17} support this claim.

Third, though we focus on the ``high-noise regime,'' we note that Theorems~\ref{thm:main_lower} and~\ref{thm:main_upper} do not require that $\sigma \to \infty$; we merely require that $\sigma$ be bounded below by a (small) constant.

The rest of this paper is devoted to the proof of the main results in Theorems~\ref{thm:main_lower} and~\ref{thm:main_upper}. 

\section{Maximum likelihood estimation}\label{sec:uniform_upper}

Let $Y_1, \dots, Y_n$ be i.i.d observations from the phase shift model~\eqref{eq:phase_shift} and consider the MLE $\tilde \theta_n$ that was defined in~\eqref{eq:ML:mixtureGaussian}. In this section, we prove our main statistical result, that is  a \emph{uniform}  upper bound on the rate of convergence of the MLE in terms of the curvature of the divergence $D(\theta \parallel \phi)$ near its minimum. Note that this analysis departs from the classical \emph{pointwise} rate of convergence for MLE that guarantees a rate of convergence $n^{-1/2}$ \emph{for each fixed choice of parameter} as $n \to \infty$. Our tools strengthen this result considerably. Indeed, we show that for reasonable choices of $\theta$, the MLE achieves a rate of $n^{1/2}$ \emph{uniformly} over all choices of $\theta$.
We refer the reader to~\cite{HeiKah15} for examples of Gaussian mixture problems where the pointwise and uniform rates of estimation differ.

The following theorem establishes an upper bound for the MLE under a general lower bound for the KL divergence for any algebraically structured model.
Our proof technique applies to any subgroup $\cG$ of $\mathrm{O}(L)$ and 
can be broadly applied to derive uniform rates of convergence for the MLE from the tight bounds on the KL divergence given in Theorem~\ref{thm:tight_KL_bound}.
In the following section, we specialize this result to obtain the minimax upper bounds for the phase shift model over $\cT_s$ that are presented in Theorem~\ref{thm:main_upper}.

From here on, positive constants may depend on $L$ unless noted otherwise.

\begin{theorem}\label{thm:non_asymptotic}
Let $\cT$ be any subspace of $\R^L$.
Assume that there exist $k \geq 1$ and positive constants $c$ and $C$ such that for all $\theta, \phi \in \cT$ satisfying $c^{-1} \leq \|\theta\| \leq c$ and $\sigma \in \R$ satisfying $\sigma \geq \|\theta\|$,
\begin{equation}
D(\theta \parallel \phi)\geq C \sigma^{-2k}\rho^2(\theta, \phi)\,.\label{EQ:universal_curvature}
\end{equation}
Then there exists positive constants $C'$ and $C''$ such that the MLE $\tilde \theta_n$ constrained to lie in $\cT$ satisfies
\begin{equation}\label{eq:theorem:non_asymptotic}
\E_\theta[\metric(\tilde \theta_n, \theta)] \leq C'\frac{\sigma^{k}}{\sqrt{n}} + C_\sigma\frac{\log n}{n}\,,
\end{equation}
uniformly over $\theta \in \cT$ satisfying $c^{-1} \leq \|\theta\| \leq c$ and $\sigma \in \R$ satisfying $\sigma \geq \|\theta\|$, where $C_\sigma\le C''\sigma^{6k-5}$.
\end{theorem}

\begin{proof}

The symbols $c$ and $C$  denote constants whose value may change from line to line.
In the rest of this proof, we write $\tilde \theta=\tilde \theta_n$ to denote the constrained MLE.
Since $\theta$ and $\tilde \theta$ are both constrained to lie in $\cT$, we  restrict all functions of this proof to this subspace without loss of generality.
By rescaling by a constant, we can assume $\|\theta\| = 1$ and $\sigma \geq 1$.

The  proof strategy is to combine control of the curvature of the function $D$ with control of the deviations of the log-likelihood function.

Define the event $\cE=\{\metric(\tilde \theta, \theta) \le \eps\}$ where $\eps$ is to be specified. Since $D$ is invariant under the action of $\cG$, we can assume without loss of generality that $\metric(\tilde \theta, \theta) = \|\tilde \theta - \theta\|$.
We first establish that on this event, $\|\tilde \theta - \theta\|$ can be controlled in terms of the metric induced by the Hessian of $D$ at $\theta$.

Fix $\theta \in \cT$ and denote by $H$ the Hessian of the function $\phi \mapsto D(\theta \parallel \phi)$
evaluated at $\phi = \theta$.
For any $u \in \R^L$, define $\|u\|_H = \sqrt{u^\top H u}$.

It follows from a Taylor expansion (Lemma~\ref{lem:d_taylor} in Appendix~\ref{supp}) that on $\cE$,
\begin{equation*}
\Big|D(\theta \parallel\tilde \theta) - \frac 1 2 \|\tilde \theta -\theta\|_H^2\Big| \leq C\frac{\|\tilde \theta -\theta\|^3}{\sigma^3}<\frac12D(\theta \parallel\tilde \theta)\,.
\end{equation*}
as long as $\ep \leq c \sigma^{3 - 2k}$ for some sufficiently small constant $c$. This yields
\begin{equation}
\frac 1 3 \|\tilde \theta - \theta\|_H^2 \leq D(\theta \, \| \, \tilde \theta) \leq \|\tilde \theta - \theta\|_H^2\,, \label{eqn:d_to_h_bound}
\end{equation}
and, by~\eqref{EQ:universal_curvature},
\begin{equation}
\|\tilde \theta - \theta\|_H^2 \geq c \sigma^{-2k} \|\tilde \theta - \theta\|^2\,. \label{eqn:rho_to_h_bound}
\end{equation}
for some constant $c$.

We now control the geometry of the log-likelihood function near $\theta$.
Define
\begin{equation*}
D_n(\theta \parallel \phi) = \frac 1 n \sum_{i=1}^n \log \frac{f_\theta}{f_\phi}(Y_i)\,,
\end{equation*}
where $Y_i$ are i.i.d from $\dist{P}_\theta$ and $f_\zeta$ is the density of $\dist{P}_\zeta, \zeta \in \R^L$.
Note that $D_n(\theta \parallel \theta) = 0$ and recall that $\tilde \theta$ minimizes $\phi \mapsto D_n(\theta \parallel \phi)
$ so that $D_n(\theta \parallel \tilde \theta) \leq 0$.

Since $\theta$ is held fixed throughout the proof, we abbreviate $D(\theta \parallel \phi)$ and $D_n(\theta \parallel \phi)$ as $D(\phi)$ and $D_n(\phi)$, respectively.

Using Taylor expansion and  $D(\theta)=D_n(\theta)=\nabla D(\theta)=0$, we get
\begin{equation*}
D(\tilde \theta) - D_n(\tilde \theta) = -\nabla D_n(\theta)^\top h + \frac 1 2 h^\top \nabla^2 (D - D_n)(\bar \theta)h\,,
\end{equation*}
where $h=\tilde \theta -\theta$ and $\bar \theta$ lies on a segment between $\tilde \theta$ and $\theta$.

For all $\zeta \in \cT$, write $H_n(\zeta)$ for the Hessian of $D_n(\phi)$ evaluated at $\phi = \zeta$, and similarly let $H(\zeta)$ be the Hessian of $D(\phi)$ evaluated at $\phi = \zeta$.

Combining the above equation with~\eqref{eqn:d_to_h_bound} and the fact that $D_n(\tilde \theta) \leq 0$ yields
\begin{equation}\label{eqn:second_order_expansion}
\frac 1 3 \|h\|_H^2 \leq D(\tilde \theta) - D_n(\tilde \theta) \leq - \nabla D_n(\theta)^\top h + \frac 1 2 h^\top (H(\bar \theta) - H_n(\bar \theta)) h\,.
\end{equation}

For the first term, we employ the bound $|\nabla D_n(\theta)^\top h| \leq \|\nabla D_n(\theta)\|_H^* \|h\|_H$, where $\|\cdot\|_H^*$ denotes the dual norm to $\|\cdot\|_H$.

To control the second, note first that by~\eqref{eqn:rho_to_h_bound}, it holds  $\|h\| \le C\sigma^k \|h\|_H$. Therefore,
\begin{equation*}
h^\top (H(\bar \theta) - H_n(\bar \theta)) h \leq C\sigma^k\|h\|_H\|h\|  \sup_{\phi \in \cB_\eps} \|H(\phi) - H_n(\phi)\|_\text{op} \,,
\end{equation*}
where $\cB_\eps:=\{\phi \in \R^L\,:\, \rho(\phi, \theta) \le \eps\}$.

Combining the above bounds and dividing by $\|h\|_H$, we get that on ${\cE}$, 
\begin{align*}
\sigma^{-k}\|h\| \leq C \|h\|_H & \leq C\|\nabla D_n(\theta)\|_{H}^* + C\sigma^k\|h\|  \sup_{\phi \in \cB_\eps} \|H(\phi) - H_n(\phi)\|_\text{op}\\
& \leq C\|\nabla D_n(\theta)\|_{H}^* + C\sigma^{k+3} \sup_{\phi \in \cB_\eps} \|H(\phi) - H_n(\phi)\|_\text{op}^2+ \sigma^{k-3}\|h\|^2\,,
\end{align*}
where we applied Young's inequality.

Since $\ep^{-1}\metric(\tilde \theta, \theta) \geq 1$ on $\cE^c$, we get
\begin{align*}
\E[\metric(\tilde \theta, \theta)] &= \E[\metric(\tilde \theta, \theta)\1_{\cE}] + \E[\metric(\tilde \theta, \theta)\1_{{\cE}^c}]\\
&\le  C\sigma^{k}\E[\|h\|_H\1_{\cE}] +\eps^{-1}{\E[\metric^2(\tilde \theta, \theta)]}\,.
\end{align*}

Choose $\ep = c \sigma^{3-2k}$ for some small constant $c$.
We obtain
\begin{equation}
\label{EQ:pr_mainbd}
\E[\metric(\tilde \theta, \theta)]  \le C\Big(\sigma^k\E\|\nabla D_n(\theta)\|_{H}^* + \sigma^{2k+3} \E\sup_{\phi \in \cB_\eps} \|H(\phi) - H_n(\phi)\|_\text{op}^2+ \sigma^{2k-3}\E[\metric^2(\tilde \theta, \theta)]\Big)
\end{equation}

It suffices to control the right side of the above inequality.
The main term is the first one.
We note that if $H$ were invertible, and hence $\|\cdot\|_H$ a genuine metric, then it is well known (see, e.g.,~\cite{HirLem01short}) that
$
\|\nabla D_n(\theta)\|_H^*=\|\nabla D_n(\theta)\|_{H^{-1}}\,.
$
In general, $H$ is not invertible, but we still have 
\begin{equation*}
\|u\|_H^* = \left\{\begin{array}{ll}
\sqrt{u^\top H^\dagger u} & \text{ if $u$ lies in the row space of $H$,} \\
\infty & \text{ otherwise,}
\end{array}\right.
\end{equation*}
where $H^\dagger$ denotes the Moore-Penrose pseudo-inverse of the matrix $H$.
The Bartlett identities state that
\begin{align*}
\E[\nabla_\phi \log f_\phi(Y_i)|_{\phi = \theta}] & = 0\,,\\
\E[(\nabla_\phi \log f_\phi(Y_i)|_{\phi = \theta}) (\nabla_\phi \log f_\phi(Y_i)|_{\phi = \theta})^\top] & = H\,,
\end{align*}
and since $\nabla D_n(\theta) = \frac 1 n \sum_{i=1}^n \nabla_\phi \log f_\phi(Y_i)|_{\phi = \theta}$ and $Y_1, \dots, Y_n$ are independent, we obtain
\begin{equation*}
\E[\nabla D_n(\theta) \nabla D_n(\theta)^\top] = \frac 1 n H\,.
\end{equation*}
In particular, $\nabla D_n(\theta)$ lies in the row space of $H$ almost surely. Jensen's inequality implies 
\begin{equation}
\label{EQ:bdGradient}
\E\|\nabla D_n(\theta)\|_{H}^* \le \Big(\tr( H^{\dagger} \E[\nabla D_n(\theta) \nabla D_n(\theta)^\top])\Big)^{1/2}  = \Big(\frac 1 n \tr(H^{\dagger} H)\Big)^{1/2}  \leq \sqrt{\frac L n}\,.
\end{equation}

For the second term, standard matrix concentration bounds can be applied to show 
\begin{equation}
\label{EQ:bdHessian}
\E\sup_{\phi \in \cB_\eps} \|H(\phi) - H_n(\phi)\|_\text{op}^2 \leq C \frac{\log n}{n\sigma^4}\,.
\end{equation}
A proof of~\eqref{EQ:bdHessian} appears as Lemma~\ref{lem:bdHessian} in Appendix~\ref{supp}.

Likewise, a standard slicing argument, Lemma~\ref{eventA} in Appenddix~\ref{supp}, implies
\begin{equation}
\label{EQ:slicing_ub}\E[\metric^2(\tilde \theta, \theta)] \le C\frac{\sigma^{4k-2}}{n}\,.
\end{equation}

Plugging~\eqref{EQ:bdGradient}, \eqref{EQ:bdHessian}, and~\eqref{EQ:slicing_ub} into~\eqref{EQ:pr_mainbd}, we get
$$
\E[\metric(\tilde \theta, \theta)]  \le C\Big(\frac{\sigma^k}{\sqrt{n}} + \frac{\sigma^{2k-1}\log n}{n}+ \frac{\sigma^{6k-5}}{n}\Big)\,,
$$
as desired.
\end{proof}

\section{Minimax upper bounds}
\label{SEC:modMLE}
In this section, we apply the results of Section~\ref{sec:uniform_upper} to the phase shift model~\eqref{eq:phase_shift}. Note that rather than the MLE, we study a constrained MLE because the lower bound~\eqref{EQ:universal_curvature} may only hold for a proper subset $\cT \subset \R^L$ in Theorem~\ref{thm:non_asymptotic}.
This phenomenon does occur in the specific case of phase shifts: the divergence $D(\phi)$ is not curved enough in directions that perturb a null Fourier coefficient of $\theta$. To overcome this limitation, we split the sample $Y_1, \ldots, Y_n$ into two parts: with the first part we estimate the support of $\hat \theta$ under Assumption~\ref{assumption} and with the second part, we compute a maximum likelihood estimator constrained to have the estimated support.

Specifically, assume for simplicity that we have a sample $Y_1, \ldots, Y_{2n}$ of size $2n$ and split it into two samples $\cY_1=\{Y_1, \ldots, Y_n\}$ and $\cY_2=\{Y_{n+1}, \ldots, Y_{2n}\}$ of equal size. 

\subsection{Fourier support estimation}

We use the first subsample $\cY_1$ to  construct a set $\tilde S$ that coincides with $\supp(\hat \theta)$ with high probability.
For any $j =1, \ldots, \lfloor L/2\rfloor $, define,
$$
M_j = \frac 1 n \sum_{i=1}^n |\widehat{(Y_i)}_j|^2 - \sigma^2\,.
$$
Recall that, by Assumption~\ref{assumption}, there exists a positive constant $c_0$ such that $|\hat \theta_j| \geq c_0$ for all $j \in \supp(\hat \theta)$.
Define the set $\tilde S$ by
$$
\tilde S=\Big\{ j \in \{1, \ldots, \lfloor L/2\rfloor\}\,:\,  M_j \geq \frac 12 c_0^2|  \Big\}
$$

The following proposition shows that  $\tilde S=\supp(\hat \theta)$ with high probability.

\begin{proposition}\label{prop:projection_construction}
There exists a positive constant $c$ depending on $c_0$ such that 
\begin{equation*}
\p[\tilde S \neq \supp(\hat \theta)] \le   2L \exp(-c n \sigma^{-4})\,.
\end{equation*}
\end{proposition}
\begin{proof}
This follows from standard concentration arguments. A full proof appears in Appendix~\ref{supp}.
\end{proof}

\subsection{Constrained MLE}

We use the second sample to construct a constrained MLE. To that end, for any $S \subset  \{1, \ldots, \lfloor L/2\rfloor\}$,  define the projection $P_S$ by
\begin{equation*}
\widehat{P_S(\phi)}_j = \left\{\begin{array}{ll}
\hat \phi_j & \text{ if $j \in S \cup -S$} \\
\hat \phi_0 & \text{ if $j = 0$} \\
0 & \text{ otherwise.}
\end{array}
\right.
\end{equation*}
The image $\im(P_{\tilde S})$ of $P_S$ is a $(2|S|+1)$-dimensional real vector space.
For convenience, write $\phi_S = P_S \phi$ for any vector $\phi \in \R^L$.

Having constructed the set $\tilde S$, we use the samples in $\cY_2$ to calculate a modified MLE $\check \theta_n$ constrained to lie in the subspace $\im(P_{\tilde S})$. To analyze the performance of this constrained MLE, we check that~\eqref{EQ:universal_curvature} holds on this subspace.

\begin{proposition}\label{prop:lambda_P_bound}
Fix  $2 \leq s \leq \lfloor L/2 \rfloor$ and $\theta \in \cT_s$. Let $S = \supp(\hat \theta)$.
If $\cT = \im(P_S)$, then there exists $C>0$ such that
for all $\sigma\ge \|\theta\|, \phi \in \cT$, it holds
\begin{equation}
D(\theta \parallel \phi)\geq C \sigma^{-4s+2}\rho^2(\theta, \phi) \,.\label{EQ:shift_universal_curvature}
\end{equation}
\end{proposition}
\begin{proof}
For the sake of exposition, we only prove~\eqref{EQ:shift_universal_curvature} for $\phi$ such that $\metric(\theta, \phi) \leq \ep_0$ for some small $\ep_0$ to be specified. The complete proof is deferred to Appendix~\ref{supp}.
In what follows, the symbols $c$ and $C$ will refer to unspecified positive constants whose value may change from line to line. By rescaling by a constant, we can assume $\|\theta\| = 1$ and $\sigma \geq 1$.

By Lemma~\ref{lem:first_moment},
\begin{equation}\label{eqn:k1_boundmt}
D(\phi) = \frac{1}{2\sigma^2}\| \E[G \theta - G \phi]\|^2 + D(\vartheta \parallel \varphi)\,,
\end{equation}
where $\vartheta = \theta - \E G\theta$ and $\varphi = \phi - \E G \phi$.

If $|\hat \theta_0 - \hat \phi_0| \geq \frac 1 2 \metric(\theta, \phi)$, then~\eqref{eqn:k1_boundmt} implies
\begin{equation*}
D(\phi) \geq \frac{1}{2\sigma^2}\| \E[G \theta - G \phi]\|^2 = \frac{1}{2\sigma^2}(\hat \theta_0 - \hat \phi_0)^2 \ge \frac{\metric^2(\theta, \phi)}{8\sigma^2} \geq \frac 1 8 \sigma^{-4s+2} \metric^2(\theta, \phi)\,.
\end{equation*}
On the other hand, if $|\hat \theta_0 - \hat \phi_0| < \frac 1 2 \metric(\theta, \phi)$, then
$$
\metric(\vartheta, \varphi)^2 = \metric(\theta_S, \phi)^2  - |\hat \theta_0 - \hat \phi_0|^2 \ge 3\metric^2(\theta, \phi)/4\,.
$$
Thus, by~\eqref{eqn:k1_boundmt}, it suffices to show that
\begin{equation*}
D(\vartheta \parallel \varphi) \ge C \sigma^{-4s+2} \metric(\vartheta, \varphi)^2\,,
\end{equation*}
for vectors $\vartheta$ and $\varphi$ satisfying $\E G \vartheta  = \E G \varphi = 0$.
In what follows, write $\metric(\vartheta, \varphi) = \ep$.
Since $D(\varphi)=D(G\varphi)$ for all $G \in \cS$, we may assume that $\|\vartheta - \varphi\| = \ep$.
We will show that there exists a small positive constant $c$ such that for some $m \leq 2 s - 1$,
\begin{equation*}
\|\Delta_m\| := \|\E[(G\vartheta)^{\otimes m} - (G\varphi)^{\otimes m}]\| \geq c \ep\,,
\end{equation*}
and the claim will follow from Theorem~\ref{thm:tight_KL_bound}.
We denote by $\kappa$ a small constant whose value will be specified.
There are two cases: either $\vartheta$ and $\varphi$ have essentially the same power spectrum (i.e., $|\hat \vartheta_k| \approx |\hat \varphi_k|$ for all $k$) or their power spectra are very different.
We will treat these two cases separately.

Recall that for each $j \in S$, by Assumptions~\ref{assumption:norm} and~\ref{assumption}, the bounds $c_0^{-1}\le| \hat \vartheta_j| \le c$ hold. Consider the polar form $\hat \varphi_j/\hat \vartheta_j = r_j e^{\mathrm{i} \delta_j}$,
where $r_j\ge 0$.

\paragraph{Case a: There exists $j \in S$ such that $|1 - r_j| \ge \kappa \eps$}
The fact that $|\hat \vartheta_j| \geq c_0^{-1}$ implies
\begin{align*}
\|\Delta_2\|^2 & = \|\E[(G \vartheta)^{\otimes 2} - (G \varphi)^{\otimes 2}]\|^2 \\
& = \sum_{k = -\lfloor L/2 \rfloor}^{\lfloor L/2 \rfloor} (|\hat \vartheta_k|^2 - |\hat \varphi_k|^2)^2 \\
& \geq (|\hat \vartheta_j|^2 - |\varphi_j|^2)^2 \\
& = |\hat \vartheta_j|^4 (1 - r_j^2)^2 \\
& \geq c_0^{-4} (1+r_j)^2(1-r_j)^2 \\
& \geq c_0^{-4} \kappa^2 \eps^2\,,
\end{align*}
so that $\|\Delta_2\| \geq c \eps$.

\paragraph{Case b: $|1 - r_j| < \kappa \ep$ for all $j \in S$}
Denote by $p$ the smallest integer in $S$ and observe that 
\begin{equation*}
\eps^2= \metric(\vartheta, \varphi)^2 = \min_{z: |z| = 1} 2 \sum_{j \in S} |1-r_jz^je^{\mathrm{i}\delta_j}|^2|\hat \vartheta_j|^2 \le C \sum_{j \in S} |1 - r_j e^{\mathrm{i}(p\delta_j-j \delta_p)/p}|^2\,,
\end{equation*}
where the inequality follows from choosing $z = e^{-\mathrm{i} \delta_p/p}$.
Therefore, there exists a coordinate $\ell \in S$ such that
\begin{equation}\label{EQ:deltaphases}
|1 - e^{\mathrm{i} (p \delta_\ell - \ell \delta_p)/p}| \geq |1 - r_\ell e^{\mathrm{i}(p \delta_\ell - \ell \delta_p)/p}| - \kappa \ep \geq c \ep\,,
\end{equation}
as long as $\kappa $ is chosen sufficiently small. In particular, $|1 - e^{\mathrm{i} (p \delta_\ell - \ell \delta_p)/p}| > 0$, so $\ell \neq p$.
Note that this fact implies that, if $|1 - r_j| < \kappa \ep$ for all $j \in S$, then $|S| \geq 2$.

Choose $m = \ell + p$.
Since $\ell, p \in S \subseteq [s]$ and $\ell \neq p$, the bound $m \leq 2s - 1$ holds.
As in the proof of Proposition~\ref{prop:generalized_matching}, we have that
\begin{align*}
\|\Delta_m\|^2 & = \sum_{j_1 + \dots + j_{m} = 0} \left|\prod_{n = 1}^{m} \hat \vartheta_{j_n}- \prod_{n = 1}^{m} \hat \varphi_{j_n}\right|^2 \\
& = \sum_{j_1 + \dots + j_{m} = 0} \left|1 - \prod_{n = 1}^{m} r_{j_n} e^{\mathrm{i} \delta_{j_n}}\right|^2 \prod_{n = 1}^{m} |\hat \vartheta_{j_n}|^2\,.
\end{align*}
Each term in the above sum is positive.
One valid solution to the equation $j_1 + \dots + j_m = 0$ is $j_1 = \dots = j_\ell = -p$ and $j_{\ell + 1} = \dots = j_m = \ell$.
We obtain
\begin{align*}
\|\Delta_m\|^2 &\ge C \left|1 - e^{i (p\delta_\ell - \ell \delta_p)} \prod_{n = 1}^{m} r_{j_n} \right|^2 \\
& \geq C |1 - e^{i (p\delta_\ell - \ell \delta_p)}|^2 - C\left|1 - \prod_{n = 1}^{m} r_{j_n}\right|^2\,.
\end{align*}
As long as $\kappa \ep$ is small enough, $\left|1 - \prod_{n = 1}^{m} r_{j_n}\right| \leq 2 m \kappa \ep$. Moreover, as long as $\ep_0 $ is chosen sufficiently small, $\delta_\ell$ and $\delta_p$ can both be chosen small enough that $|p \delta_\ell - \ell \delta_p| \leq 1$, in which case it holds
$$
|1 - e^{i (p\delta_\ell - \ell \delta_p)}|^2 \ge |1 - e^{i (p \delta_\ell - \ell \delta_p)/p}|^2 \ge c^2\eps^2\,,
$$
where the last inequality follows from~\eqref{EQ:deltaphases}.
So for $\ep_0$ and $\kappa$ chosen sufficiently small, this proves the existence of an $m \leq 2s -1$ for which $\|\Delta_m\| \geq c \ep$.

\end{proof}

\subsection{Proof of Theorem~\ref{thm:main_upper}}\label{sec:upper_proof}
Define $\cR=\{\tilde S=\supp(\hat \theta)\}$ and observe that
$$
\E[\metric(\check \theta_{n}, \theta)] = \E[\metric(\check \theta_{n}, \theta)\1_{\cR}] + \E[\metric(\check \theta_{n}, \theta)\1_{\cR^c}]\,.
$$
The first term is controlled by combining Proposition~\ref{prop:lambda_P_bound} and Theorem~\ref{thm:non_asymptotic} to get
$$
\E[\metric(\check \theta_{n}, \theta)\1_{\cR}]\le C\frac{\sigma^{2s-1}}{\sqrt{n}} + C_\sigma\frac{\log n}{n}\,,
$$
where $C_\sigma \le C\sigma^{12s-11}$.

To bound the second term, we use the Cauchy-Schwarz inequality and Proposition~\ref{prop:projection_construction} to get
$$
\E[\metric(\check \theta_{n}, \theta)\1_{\cR^c}]\le 2L \exp(-c n \sigma^{-4}) \sqrt{\E[\metric(\check \theta_{n}, \theta)^2]} \leq C \frac{\sigma^4}{n} \sqrt{\E[\metric(\check \theta_{n}, \theta)^2]}\,.
$$
We now show that $\E[\metric(\check \theta_{n}, \theta)^2]$ is bounded uniformly over all choices of $\tilde S$ by a constant multiple of $\sigma^2$ using a similar slicing argument as the one employed in the proof of Lemma~\ref{eventA} in Appendix~\ref{supp}.

By the triangle inequality,
\begin{equation*}
\metric(\check \theta_{n}, \theta) \leq \metric(\check \theta_{n}, \theta_{\tilde S}) + \metric(\theta_{\tilde S}, \theta) \leq \metric(\check \theta_{n}, \theta_{\tilde S}) + 1\,.
\end{equation*}
By Lemma~\ref{lem:far_quadratic} in Appendix~\ref{supp}, when $\metric(\check \theta_{n}, \theta_{\tilde S}) \geq 3 \sqrt 2 \sigma$, the divergence satisfies $D(\theta_{\tilde S} \parallel \check \theta_{n}) \geq c \sigma^{-2} \metric(\check \theta_{n}, \theta_{\tilde S})$.
We therefore have
$$
\metric(\check \theta_{n}, \theta_{\tilde S})^2\le 18 \sigma^2 + c^{-1} \sigma^2 D(\theta_{\tilde S} \parallel \check \theta_{n}) \le (C^\circ \sigma)^2(1+ \mathfrak{G}_n(\theta_{\tilde S} \parallel \check \theta_{n}))\,,
$$
for some constant $C^\circ$, where
\begin{equation*}
\mathfrak{G}_n(\theta_{\tilde S} \parallel \check \theta_{n}) = D(\theta_{\tilde S} \parallel \check \theta_{n}) - D_n(\theta_{\tilde S} \parallel \check \theta_{n})
\end{equation*}
and where we have used the fact that $D_n(\theta_{\tilde S} \parallel \check \theta_{n}) \leq D_n(\theta_{\tilde S} \parallel \theta_{\tilde S}) = 0$.

For $j \ge 0$, define $T_j=\{\phi \in \R^d\,:\, 2^j\sigma \le \rho(\phi, \theta_{\tilde S})\le 2^{j+1}\sigma\}$ and  let $J$ be such that $C^\circ \le 2^J \le 2C^\circ$. Observe that
\begin{align*}
\E[\metric(\check \theta_{n}, \theta_{\tilde S})^2]&\le 4 (C^\circ \sigma)^2 + \sum_{j\ge J} \E[\metric(\check \theta_{n}, \theta_{\tilde S})^2\1(\check \theta_{n} \in T_j)]\\
&\le  4 (C^\circ \sigma)^2 + \sigma^2\sum_{j>J} 2^{2j+2}\p[\sup_{\phi \in T_j}\mathfrak{G}_n(\theta_{\tilde S} \parallel \phi)>2^{2j}]\\
&\le 4 (C^\circ \sigma)^2 + C \sigma^2 \sum_{j\ge 0} 2^{2j}\exp(-C2^{2j})\le C \sigma^2\,,
\end{align*}
where we used~\eqref{eqn:divergence_tail_bound} from Appendix~\ref{supp} in the third inequality. 
We obtain
\begin{equation*}
\E[\metric(\check \theta_{n}, \theta)^2] \leq 2 \E[\metric(\check \theta_{n}, \theta_{\tilde S})^2] + 2 \leq C\sigma^2\,.
\end{equation*}

We have established that
$$
\E[\metric(\check \theta_{n}, \theta)] \le C\Big(\frac{\sigma^{2s-1}}{\sqrt{n}}+\sigma^{12s - 11}\frac{\log n}{n}+\frac{\sigma^5}{n} \Big)\,,
$$
which, since $s \geq 2$, completes the proof of Theorem~\ref{thm:main_upper}.

\section{Minimax lower bounds}\label{sec:multi}
Our minimax lower bounds rely ultimately on Le Cam's classical two-point testing method~\cite{LeC73}. For this reason, our lower bounds do not capture the optimal dependence in $L$ but only in $\sigma$ and $n$.  In particular, the version that we use requires an upper bound on the KL divergence, which can be obtained using Theorem~\ref{thm:tight_KL_bound} and a moment matching argument.

\subsection{Moment matching}
Theorem~\ref{thm:tight_KL_bound} implies that $\dist{P}_\theta$ and $\dist{P}_\phi$ are hard to distinguish when the quantities
\begin{equation*}
\|\Delta_m\| = \|\E[(G\theta)^{\otimes m}] - \E[(G\phi)^{\otimes m}]\|
\end{equation*}
vanish for $m \in [k]$. In this section, we show that, in the phase shift model, $\Delta_m, m \in [k]$ can be made to vanish for large $k$ by appropriately choosing the support of the Fourier transforms $\hat \theta$ and $\hat \phi$.

Before stating our main results, we first give a brief sketch of the technique. As we show in the proof of Proposition~\ref{prop:generalized_matching} below, the tensor $\E[(G \theta)^{\otimes m}]$ has a simple form in the Fourier basis:
\begin{equation*}
\E[(\widehat{G \theta})^{\otimes m}]_{j_1 \dots j_m} = \left\{\begin{array}{ll}
\hat \theta_{j_1} \cdots \hat \theta_{j_m} & \text{ if $j_1 + \dots + j_m = 0$,} \\
0 & \text{ otherwise.}
\end{array}\right.
\end{equation*}

For example, the first moment tensor $\E[G \theta]$ contains only the term $\hat \theta_0$, and the second moment tensor $\E[(G \theta)^{\otimes 2}]$ contains the term $\hat \theta_j \hat \theta_{-j}$ for each index~$j$. Since $\theta$ has real entries, $\hat \theta_j \hat \theta_{-j} = |\hat \theta_j|^2$, so that $\E[(G \theta)^{\otimes 2}]$ contains enough information to reconstruct the magnitudes of the Fourier coefficients, but not their phases.
This implies that if $\theta$ and $\phi$ satisfy $|\hat \theta_j| = |\hat \phi_j|$ for all $j$, then $\|\E[(G\theta)^{\otimes 2}] - \E[(G\phi)^{\otimes 2}]\| = 0$.

To exhibit two signals whose higher moments also match, we employ the following idea: if a tuple $(j_1, \dots, j_m)$ is of the form $(j_1, -j_1, j_2, -j_2, \dots)$, then $\E[(\widehat{G \theta})^{\otimes m}]_{j_1 \dots j_m} = |\hat \theta_{j_1}|^2  \cdots |\hat \theta_{j_m}|^2$. In other words, this entry of the $m$th moment tensor also only depends on the magnitudes of the Fourier coefficients of $\theta$ and not on their phases. Therefore, if the only nonzero entries of $\E[(\widehat{G \theta})^{\otimes m}]$ and $\E[(\widehat{G \phi})^{\otimes m}]$ correspond to tuples of this form and if the magnitudes of the Fourier coefficients of $\theta$ and $\phi$ agree, then $\|\E[(G\theta)^{\otimes m}] - \E[(G\phi)^{\otimes m}]\| = 0$.

This argument is formalized in Proposition~\ref{prop:generalized_matching}.

\begin{proposition}\label{prop:generalized_matching}
Fix $2 \leq s \leq \lfloor L/2 \rfloor$ and let $\theta, \phi \in \R^L$ satisfy
\begin{equation*}
\hat \theta_j = \hat \phi_j = 0 \quad \text{for $j \notin \{\pm(s-1), \pm s \}$}
\end{equation*}
and
\begin{equation*}
|\hat \theta_{j}| = |\hat \phi_{j}| \quad \text{for $j \in \{\pm(s-1), \pm s\}$.}
\end{equation*}
If $G$ is drawn uniformly from $\cS$, then for any $m =1, \dots, 2s-2$,
\begin{equation*}
\E[(G \theta)^{\otimes m}] = \E[(G \phi)^{\otimes m}] 
\end{equation*}
\end{proposition}
\begin{proof}
Fix $m \leq 2s-2$.
Since $\E[(G \theta)^{\otimes m}]$ and $\E[(G \phi)^{\otimes m}]$ are symmetric tensors, to show that they are equal it suffices to show that
\begin{equation*}
\langle \E[(G \theta)^{\otimes m}], u^{\otimes m}\rangle = \langle \E[(G \phi)^{\otimes m}], u^{\otimes m}\rangle \qquad \forall\, u \in \R^L
\end{equation*}
or equivalently, that
\begin{equation}
\label{EQ:equaltensors}
\E[(u^\top G \theta)^m] = \E[(u^\top G \phi)^m]\qquad \forall\, u \in \R^L\,.
\end{equation}

Consider the set $\cP = \{\zeta \,:\, |\hat \zeta_j| = |\hat \theta_j| \,, \ \forall\, j\}$ and note that $\theta, \phi \in \cP$.
We  show that the function $\zeta \mapsto \E[(u^\top G \zeta)^m]$ is constant on $\cP$, which readily yields~\eqref{EQ:equaltensors}.
For a fixed shift $G_z \in \cF$, we obtain
\begin{equation*}
u^\top G_z \zeta = \langle \hat u, \widehat{G_z \zeta}\rangle = \sum_{j = -\lfloor L/2 \rfloor}^{\lfloor L/2 \rfloor} \hat u_{-j} \hat \zeta_j z^{j}\,,
\end{equation*}
so
\begin{equation*}
(u^\top G_z \zeta)^m = \sum_{j_1, \dots, j_m = -\lfloor L/2 \rfloor}^{\lfloor L/2 \rfloor} z^{j_1 + \dots + j_m} \prod_{n = 1}^m \hat u_{-j_n} \hat \zeta_{j_n}\,.
\end{equation*}

Taking expectations with respect to a uniform choice of $z$ yields
\begin{equation}
\label{EQ:sumofji}
\E[(u^\top G_z \zeta)^m]  = \sum_{j_1 + \dots + j_m = 0} \prod_{n = 1}^m \hat u_{-j_n} \hat \zeta_{j_n}\,,
\end{equation}
where the sums are over all choices of coordinates $j_1, \dots, j_m \in \{-\lfloor L/2 \rfloor, \dots, \lfloor L/2 \rfloor\}$ whose sum is $0$.

The Fourier transform of $\zeta$ is supported only on coordinates $\pm (s-1)$ and $\pm s$, so we may restrict our attention to sums involving only those coordinates.
Suppose $j_1 + \dots + j_m = 0$.
Define
\begin{align*}
\alpha & = |\{i : j_i = s-1\}| &
\beta & = |\{i : j_i = -(s-1)\}| \\
\gamma & = |\{i : j_i = s\}| &
\delta & = |\{i : j_i = -s\}| 
\end{align*}

By assumption $j_1 + \dots + j_m = 0$, so the tuple $(\alpha, \beta, \gamma, \delta)$ is a solution to
$$
\alpha (s-1) + \beta (-(s-1)) + \gamma (s) + \delta (-s)  = 0
$$
or, equivalently,
\begin{align*}
(\alpha-\beta) (s-1) + (\gamma-\delta) s & = 0 \nonumber
\end{align*}

Since $s-1$ and $s$ are coprime, $(\alpha - \beta)$ and $(\gamma - \delta)$ must be multiples of $s$ and $s-1$, respectively.
Since $|\alpha - \beta| + |\gamma - \delta| \leq m < 2s - 1$, in fact $\alpha - \beta = \gamma - \delta = 0$.

Therefore the only $m$-tuples $(j_1, \ldots, j_m)$ that appear in the sum on the right-hand side of~\eqref{EQ:sumofji} are those in which $+(s-1)$ and $-(s-1)$ occur an equal number of times and $+s$ and $-s$ occur an equal number of times.
For such $m$-tuples, the product $\prod_{n = 1}^m \hat u_{-j_n} \hat \zeta_{j_n}$ can be reduced to a product of terms of the form $\hat u_{-(s-1)}\hat u_{s-1}\hat \zeta_{s-1} \hat \zeta_{-(s-1)}$ and $\hat u_{-s}\hat u_{s}\hat \zeta_s \hat \zeta_{-s}$.
Since $u$ and $\zeta$ are real vectors, $\hat u_j \hat u_{-j} = |u_j|^2$ and $\hat \zeta_j \hat \zeta_{-j} = |\hat \zeta_{j}|^2$ for all $j =  -\lfloor L/2 \rfloor, \dots, \lfloor L/2 \rfloor$, so
\begin{equation*}
\prod_{n = 1}^m \hat u_{-j_n} \hat \zeta_{j_n} = (|\hat u_{s-1}|^2 |\hat \zeta_{s-1}|^2)^{\alpha+\beta} (|\hat u_s|^2 |\hat \zeta_s|^2)^{\gamma+\delta}\,.
\end{equation*}
This quantity depends only on the moduli $|\hat \zeta_s|$ and $|\hat \zeta_{s-1}|$, hence it is the same for all $\zeta \in \cP$. This completes the proof of~\eqref{EQ:equaltensors} and therefore the proof of the proposition.
\end{proof}

\subsection{Proof of Theorem~\ref{thm:main_lower}}
Fix $n \geq 1$.
We will select $\phi, \theta_n \in \cT_s$ such that $\rho(\phi, \theta_n) \geq c \big(\frac{\sigma^{(2s-1) \vee (s+1)}}{\sqrt n} \wedge 1\big)$ for some small universal constant $c > 0$ but $D(\dist{P}_{\theta_n}^n \parallel \dist{P}_{\phi}^n) \leq \frac 12$. The bound will then follow from standard techniques.

If $s = 0$, then let $\phi = 0$ and $\theta_n = \frac{\sigma}{\sqrt{nL
}} \bone$, where $\bone$ denotes the all-ones vector of $\R^L$. Note that $\metric(\theta_n, \phi) = \sigma/\sqrt n$.
Moving to the Fourier domain, we have $(\widehat{\theta_n})_0 = \sigma/\sqrt{n}$ and $(\widehat{\theta_n})_j = 0$ for $j \neq 0$, so that $\theta_n, \phi \in \cT_0$.
By Lemma~\ref{lem:first_moment},
\begin{equation*}
D(\dist{P}_{\theta_n}^n \parallel \dist{P}_{\phi}^n) = \frac{n\sigma^2}{2 n\sigma^2} = \frac{1}{2}\,.
\end{equation*}

If $s = 1$, let $\phi$ be given by
\begin{equation*}
\hat \phi_j = 
\left\{ \begin{array}{ll}
1/\sqrt{2} & \text{if $j  \in \{\pm 1\}$,} \\
0 & \text{otherwise.}\end{array}\right.
\end{equation*}
and let $\theta_n$ satisfy 
\begin{equation*}
(\widehat{\theta_n})_j =
\left\{ \begin{array}{ll}
1/\sqrt{2} +  c_1 \frac{\sigma^2}{\sqrt{2n}} & \text{if $j \in \{\pm 1\}$,} \\
0 & \text{otherwise,}\end{array}\right.
\end{equation*}
for some constant $c_1 > 0$ to be specified.
Clearly $\theta_n, \phi \in \cT_1$, and $\metric(\theta_n, \phi) = c_1 \frac{\sigma^2}{\sqrt n}$.

Theorem~\ref{thm:tight_KL_bound} implies that
\begin{equation*}
D(\dist{P}_{\theta_n}^n \parallel \dist{P}_{\phi}^n) \leq \overline C \frac{n\metric(\theta_n, \phi)^2}{\sigma^4} \leq \frac{1}{2}\,,
\end{equation*}
by choosing $c_1$ small enough.

Finally, suppose $s \geq 2$.
Fix $z = e^{\mathrm{i} \delta}$ for $\delta = c_1 (\sigma^{2s-1}/\sqrt{n} \wedge 1)$ for some positive constant $c_1 \leq 1$ to be specified.
Let $\phi$ be given by
\begin{equation*}
\hat \phi_j  =
\left\{ \begin{array}{ll}
1/2 & \text{if $j  \in \{\pm(s-1), \pm s \}$,} \\
0 & \text{otherwise.}
\end{array}\right.
\end{equation*}
Let $\theta_n$ be given by
\begin{align*}
(\widehat{\theta_n})_j  =
\left\{ \begin{array}{ll}
1/2 & \text{if $j \in \{\pm(s-1)\}$,} \\
z/2 & \text{ if $j = s$,} \\
z^*/2 & \text{ if $j = -s$,} \\
0 & \text{otherwise.}
\end{array}\right.\end{align*}
Note that $\theta_n$ and $\phi$ both lie in $\cT_s$.

For any unit complex number $w \in U(1)$, we have
\begin{equation*}
\|\theta_n - G_w \phi\|^2 = \frac 12 (\|1 - w\|^2 + \|z - w\|^2) \ge \frac14 \|1 - z\|^2\,.
\end{equation*}
So
\begin{equation*}
\frac 1 4 \|1-z\|^2 \leq \rho^2(\theta_n, \phi) \leq \|\theta_n - \phi\|^2 = \frac 12 \|1 - z\|^2\,,
\end{equation*}
and under the assumption that $\delta \leq 1$, we have $\frac 12 \delta^2 \leq \|1 - z\|^2 \leq \delta^2$.

Therefore
\begin{equation*}
c \delta \leq \rho(\theta_n, \phi) \leq C \delta
\end{equation*}
for absolute positive constants $c$ and $C$.

Theorem~\ref{thm:tight_KL_bound} and Proposition~\ref{prop:generalized_matching} imply that\begin{equation*}
D(\dist{P}_{\theta_n}^n \parallel \dist{P}_{\phi}^n) \le \overline{C} n \sigma^{-4s+2}\metric^2(\phi_n, \tau)\le \frac12\,,
\end{equation*}
by taking $c_1$ small enough.

In all three cases, we have a bound $D(\dist{P}_{\theta_n}^n \parallel \dist{P}_{\phi}^n) \le 1/2$
for $\theta_n, \phi \in \cT_s$ satisfying $\metric(\theta_n, \phi) \geq c (\sigma^{(2s -1) \wedge (s+1)}/\sqrt n \wedge 1)$ for some constant $c > 0$.
Using standard minimax lower bound techniques~\cite{Tsy09}, we get the desired result.

\appendix

\section{Information geometry for algebraically structured models}\label{sec:tight_bounds}
Our proof techniques rely on understanding the curvature of the Kullback-Leibler divergence around its minimum, which is known to control the information geometry of the problem.
In this section, we obtain precise bounds on the divergence $D(\dist{P}_\theta \parallel \dist{P}_\phi)$ for pairs of signals $\theta$ and $\phi$ for any choice of a subgroup $\cG$ of the orthogonal group $\mathrm{O}(L)$.

We extend the approach of~\cite{CaiLow11} to bound the divergence between $D(\dist{P}_\theta \parallel \dist{P}_\phi)$ in terms of the Hilbert-Schmidt distance between the moment tensors $\E[(G \theta)^{\otimes m}]$ and $\E[(G \phi)^{\otimes m}]$. Recall that $G$ is uniformly distributed over $\cG$.
In what follows, we write
\begin{equation*}
\Delta_m := \E[(G\theta)^{\otimes m}] - \E[(G \phi)^{\otimes m}]\,.
\end{equation*}
Our results imply that when $\sigma$ is bounded below by a constant, the divergence can be bounded above and below by an infinite series of the form $\sum_m \frac{c^m\|\Delta_m\|^2}{\sigma^{2m} m!}$.
We note that the assumption that $\sigma$ be bounded below is essential: when $\sigma \to 0$, it is not hard to show that $D(\theta \parallel \phi) = \frac{\metric^2(\phi, \theta)}{2 \sigma^2} + o(\sigma^{-2})$.

For convenience, we write $D(\theta \parallel \phi)$ for $D(\dist{P}_\theta \parallel \dist{P}_\phi)$.
We begin by establishing the effect of the first moments $\E G \theta$ and $\E G \phi$ on $D(\theta \parallel \phi)$.
\begin{lemma}\label{lem:first_moment}
If $\vartheta = \theta - \E G \theta$ and $\varphi = \phi - \E G \phi$, then
\begin{equation*}
D(\theta \parallel \phi) =D(\vartheta \parallel \varphi)+ \frac{1}{2\sigma^2} \|\Delta_1\|^2 \,.
\end{equation*}
\end{lemma}
Lemma~\ref{lem:first_moment} implies that it suffices to bound $D(\theta \parallel \phi)$ for vectors $\theta$ and $\phi$ satisfying $\E G \theta = \E G \phi = 0$, which we accomplish in the following theorem.

\begin{theorem}\label{thm:tight_KL_bound}
Let $\theta, \phi$ in $\R^L$ satisfy $3 \metric(\theta, \phi) \leq \|\theta\| \le \sigma$ and $\E G \theta = \E G \phi = 0$.
For any $k \geq 1$, there exist universal constants $\underline C$ and $\overline C$ such that
\begin{equation*}
\underline C \sum_{m=1}^{\infty} \frac{\|\Delta_m\|^2}{(\sqrt 3\sigma)^{2m} m!} \leq D(\theta \parallel \phi) \leq 2 \sum_{m=1}^{k-1} \frac{\|\Delta_m\|^2}{\sigma^{2m} m!} + \overline C\frac{\|\theta\|^{2k-2}\metric(\theta, \phi)^2}{\sigma^{2k}} \,.
\end{equation*}
\end{theorem}
In particular, Theorem~\ref{thm:tight_KL_bound} implies that if $\|\Delta_m\| = 0$ for $m = 1, \dots, k-1$ and $\|\Delta_k\| \geq c \metric(\theta, \phi)$ for some constant $c$, then $D(\theta \,\|\, \phi)$ is of order $\sigma^{-2k} \metric^2(\theta, \phi)$.

\begin{proof}[Proof of Lemma~\ref{lem:first_moment}]
We first prove the following simple expression:
\begin{equation*}
D(\dist{P}_\theta \parallel \dist{P}_\phi) = D(\theta \parallel \phi) = \frac{1}{2 \sigma^2}(\|\phi\|^2 - \|\theta\|^2) + \E \log \frac{\E [e^{\frac{1}{\sigma^2} (\theta + \sigma \xi)^\top G \theta} \mid \xi ]}{\E [e^{\frac{1}{\sigma^2} (\theta + \sigma \xi)^\top G \phi} \mid \xi ]}\,,
\end{equation*}
where $\xi\sim \cN(0,I_L)$ and $G \in \cG$ is uniform and independent of $\xi$.

This claim follows directly from the definition of divergence.
Let $\sg$ the density of a standard Gaussian random variable with respect to the Lebesgue measure on $\R^L$. It holds
\begin{align*}
\frac{d\dist{P}_\theta}{d\dist{P}_\phi}(y) & = \frac{\E[\sg((y - G\theta)/\sigma)]}{\E[\sg((y - G\phi)/\sigma)]} \\
& = \frac{\E\left[\exp\left(-\frac{1}{2\sigma^2}(\|y\|^2 - 2 y^\top G \theta + \|G \theta\|^2)\right)\right]}{\E\left[\exp\left(-\frac{1}{2\sigma^2}(\|y\|^2 - 2 y^\top G \phi + \|G \phi\|^2)\right)\right]} \\
& = \exp\left(\frac{1}{2\sigma^2}(\|\phi\|^2 - \|\theta\|^2)\right)\frac{\E\left[\exp\left(\frac{1}{\sigma^2}y^\top G \theta\right)\right]}{\E\left[\exp\left(\frac{1}{\sigma^2} y^\top G \phi\right)\right]}\,,
\end{align*}
since $G$ is orthogonal.
Hence, if $Y \sim \dist{P}_\theta$, we have
\begin{align*}
D(\theta \parallel \phi) & = \E\log \frac{d\dist{P}_\theta}{d\dist{P}_\phi}(Y) = \frac{1}{2\sigma^2}(\|\phi\|^2 - \|\theta\|^2) + \E \log \frac{\E [e^{\frac{1}{\sigma^2} Y^\top G \theta} \mid Y ]}{\E [e^{\frac{1}{\sigma^2} Y^\top G \phi} \mid Y]}\,.
\end{align*}
Note that we can write $Y = G' \theta + \sigma \xi$ for a standard Gaussian vector $\xi$ and $G' \in \cG$ an independent copy of $G$.
Since $G'\in O(L)$, $Y$ has the same distribution as $G'(\theta + \sigma \xi)$.
If $G$ and $G'$ are independent and uniform, then $(G')^\top G$ has the same distribution as $G$, so
\begin{equation*}
Y^\top G \theta \stackrel{d}{=} (\theta + \sigma \xi)^\top G \theta\,,
\end{equation*}
where the above equality holds in distribution.
It yields
\begin{equation}
\label{EQ:KLexp}
D(\theta \parallel \phi) = \frac{1}{2\sigma^2}(\|\phi\|^2 - \|\theta\|^2) + \E \log \frac{\E [e^{\frac{1}{\sigma^2} (\theta + \sigma \xi)^\top G \theta} \mid \xi ]}{\E [e^{\frac{1}{\sigma^2} (\theta + \sigma \xi)^\top G \phi} \mid \xi]}\,.
\end{equation}

We now turn to the proof of Lemma~\ref{lem:first_moment}.
For convenience write $\bar \theta = \E G \theta$ and $\bar \phi = \E G \phi$.
These vectors satisfy $G \bar \theta = \bar \theta$ and $G \bar \phi = \bar \phi$ almost surely and $\vartheta^\top\bar \theta=0$. 
Hence, almost surely,
\begin{align*}
(\theta + \sigma \xi)^\top G \theta & = (\bar \theta + \vartheta + \sigma \xi)^\top G (\bar \theta + \vartheta) \\
& = (\vartheta + \sigma \xi)^\top G \vartheta + (\bar \theta + \sigma \xi)^\top \bar \theta\,,
\end{align*}
and similarly
\begin{equation*}
(\theta + \sigma \xi)^\top G \phi = (\vartheta + \sigma \xi)^\top G \varphi + (\bar \theta + \sigma \xi)^\top \bar \phi\,.
\end{equation*}

Plugging these quantities into~\eqref{EQ:KLexp} yields
\begin{align*}
D(\theta &\parallel \phi)  = \frac{1}{2\sigma^2}(\|\phi\|^2 - \|\theta\|^2) + \E \log \frac{\E [e^{\frac{1}{\sigma^2} (\vartheta + \sigma \xi)^\top G \vartheta + (\bar \theta + \sigma \xi)^\top \bar \theta} \mid \xi ]}{\E [e^{\frac{1}{\sigma^2} (\vartheta + \sigma \xi)^\top G \varphi + (\bar \theta + \sigma \xi)^\top \bar \phi} \mid \xi]} \\
& = \frac{1}{2\sigma^2}(\|\phi\|^2 - \|\theta\|^2) + \frac{1}{\sigma^2}\E (\bar \theta + \sigma \xi)^\top (\bar \theta - \bar \phi) + \E \log \frac{\E [e^{\frac{1}{\sigma^2} (\vartheta + \sigma \xi)^\top G \vartheta} \mid \xi ]}{\E [e^{\frac{1}{\sigma^2} (\vartheta + \sigma \xi)^\top G \varphi} \mid \xi]} \\
& = \frac{1}{2\sigma^2}(\|\bar \phi\|^2 - \|\bar \theta\|^2) + \frac{1}{\sigma^2} (\|\bar \theta\|^2 - \bar \theta^\top \bar \phi) + D(\vartheta \parallel \varphi) \\
& = \frac{1}{2 \sigma^2}\|\E[G \theta - G \phi]\|^2 + D(\vartheta \parallel \varphi)\,.
\end{align*}
\end{proof}
\begin{proof}[Proof of Theorem~\ref{thm:tight_KL_bound}]
If $\theta = 0$, then the conditions of the theorem imply that $\phi = 0$, so the statement is vacuous. We therefore assume $\theta \neq 0$.
The divergence $D(\theta \parallel \phi)$ and the moment tensors $\E[(G \phi)^{\otimes m}]$ are unaffected if we replace $\phi$ by $G_0 \phi$ for any $G_0 \in \cG$.
Hence without loss of generality, we can assume that $\|\theta - \phi\| = \metric(\theta, \phi) =: \ep$.
Moreover, the quantity $D(\theta \parallel \phi)$ and the bounds in question are all unaffected upon replacing $\theta$, $\phi$, and $\sigma$ by $\theta/\|\theta\|$, $\phi/\|\theta\|$, and $\sigma/\|\theta\|$, respectively, so we assume in what follows that $\|\theta\| = 1$ and $\sigma \geq 1$.

We first prove the upper bound.
Denote by $\sg$ the density of a standard $L$-dimensional Gaussian random variable.
For all $\zeta \in \R^L$, let $f_\zeta$ denote the density of $\dist{P}_\zeta$ defined in~\eqref{eq:MRA:model}.
Recall that in this model, $G$ is drawn uniformly from the Haar measure on $\cG$.
Then, for any $y \in \R^L$ we have
\begin{equation*}
f_\zeta(y) = \E \frac{1}{\sigma^d} \sg(\sigma^{-1}(y - G \zeta)) = \frac{1}{\sigma^d} \sg(\sigma^{-1} y) e^{-\frac{\|\zeta\|^2}{2 \sigma^2}}\E e^{\frac{y^\top G \zeta}{\sigma^2}}\,.
\end{equation*}
Let $\chi^2(\theta, \phi)$ denote the $\chi^2$-divergence between $\dist{P}_\theta$ and $\dist{P}_\phi$, defined by
\begin{equation*}
\chi^2(\theta, \phi) = \int \frac{(f_\theta(y) - f_\phi(y))^2}{f_\theta(y)} \mathrm dy\,.
\end{equation*}

Since $\E G \theta = 0$ by assumption, Jensen's inequality implies
\begin{equation*}
f_\theta(y) \geq \frac{1}{\sigma^d} \sg(\sigma^{-1} y) e^{-\frac{\|\theta\|^2}{2\sigma^2}}e^{\E \frac{y^\top G \theta}{\sigma^2}} = \frac{1}{\sigma^d} \sg(\sigma^{-1} y) e^{-\frac{\|\theta\|^2}{2\sigma^2}} \,.
\end{equation*}

Hence
\begin{equation*}
\frac{(f_\theta(y) - f_\phi(y))^2}{f_\theta(y)} \leq e^{\frac{\|\theta\|^2}{2 \sigma^2}} \big(e^{-\frac{\|\theta\|^2}{2 \sigma^2}}\E e^{\frac{y^\top G \theta}{\sigma^2}} - e^{-\frac{\|\phi\|^2}{2 \sigma^2}}\E e^{\frac{y^\top G \phi}{\sigma^2}}\big)^2 \big(\frac{1}{\sigma^d} \sg(\sigma^{-1} y)\big)\,.
\end{equation*}
Integrating this quantity with respect to $y$ yields a bound on the $\chi^2$ divergence. Let $\xi \sim \cN(0, I_d)$ and observe that
\begin{align*}
\chi^2(\theta, \phi) & \leq \E \Big[e^{\frac{\|\theta\|^2}{2 \sigma^2}} \big(e^{-\frac{\|\theta\|^2}{2 \sigma^2}}\E\big[ e^{\frac{\sigma \xi^\top G \theta}{\sigma^2}}\big|\xi\big] - e^{-\frac{\|\phi\|^2}{2 \sigma^2}}\E\big[ e^{\frac{\sigma \xi^\top G \phi}{\sigma^2}}\big|\xi\big]\big)^2 \Big]\\
& =  \E \Big[ e^{\frac{\|\theta\|^2}{2 \sigma^2}} \big(e^{-\frac{\|\theta\|^2}{ \sigma^2}}e^{\frac{\xi^\top (G + G')\theta}{\sigma}} - 2 e^{-\frac{\|\theta\|^2 + \|\phi\|^2 }{2 \sigma^2}} e^{\frac{\xi^\top (G \theta + G' \phi)}{\sigma}}\\
&\le 2\E\big[  e^{\frac{(G'\theta)^\top G \theta}{\sigma^2}} - 2 e^{\frac{(G'\phi)^\top G \theta}{\sigma^2}} +e^{\frac{(G'\phi)^\top G\phi}{\sigma^2}}\big]
\end{align*}
where $G' \in \cG$ is an independent copy of $G$ and we used the bound $e^{\|\theta\|^2/(2\sigma^2)} \le 2$  that holds for  $\sigma \geq 1$ and $\|\theta\|^2 \leq 1$.

The random variables $G \theta$ and $G \phi$ have moment generating functions that converge in a neighborhood of the origin, hence
\begin{align*}
\chi^2(\theta, \phi) & \le \sum_{m \geq 0} \frac{2}{\sigma^{2m} m!} \E\left[((G'\theta)^\top G \theta)^m - 2 ((G'\phi)^\top G \theta)^m + ((G'\phi)^\top G\phi)^m\right] \\
& = \sum_{m \geq 0} \frac{2}{\sigma^{2m} m!} \|\E[(G \theta)^{\otimes m}]\|^2 - 2 \langle \E[(G \theta)^{\otimes m}], \E[(G \phi)^{\otimes m}]\rangle + \|\E[(G \phi)^{\otimes m}]\|^2 \\
& = \sum_{m \geq 0} \frac{2}{\sigma^{2m} m!} \left\| \Delta_m\right\|^2 \\
 &\leq 2\sum_{m =1}^{k-1} \frac{\|\Delta_m\|^2}{\sigma^{2m}m!} + 24 \sum_{m \geq k} \frac{2^{m}}{\sigma^{2m}m!} \ep^2
  \leq 2\sum_{m =1}^{k-1} \frac{\|\Delta_m\|^2}{\sigma^{2m}m!} + 24 e^{2} \frac{\ep^2}{\sigma^{2k}}\,,
\end{align*}
where the penultimate inequality follows from Lemma~\ref{lem:tensor_bound} in Appendix~\ref{supp}.
The bound follows upon applying the inequality $D(\theta \parallel \phi) \leq \chi^2(\theta, \phi)$~\cite{Tsy09}.

We now turn to the lower bound.
Recall that the Hermite polynomials $\{h_k(x)\}_{k \geq 0}$ satisfy the following three properties~\cite{Sze75}: %
\begin{enumerate}
\item The function $h_k(x)$ is a degree-$k$ polynomial.
\item The functions $\{h_k\}_{k \geq 0}$ form an orthogonal basis of of $L_2(\gamma)$, where $\gamma$ denotes the standard Gaussian measure on $\R$, with $\|h_k\|_\mu^2 = k!$.
\item If $Y \sim \cN(\mu, 1)$, then $\E[h_k(Y)] = \mu^k$.
\end{enumerate}

Given a multi-index $\alpha \in \N^L$, define the multivariate Hermite polynomial $h_\alpha$ by
\begin{equation*}
h_\alpha(x_1, \dots, x_L) = \prod_{i=1}^L h_{\alpha_i}(x_i)\,.
\end{equation*}
The multivariate Hermite polynomials form an orthonormal basis for the space $\R[x_1, \dots, x_L]$ of $L$-variate  polynomial functions with respect to the inner product over $L_2(\gamma^{\otimes L})$.

Given $y \in \R^L$ and $m \geq 1$, denote by $H_m(y)$ the order-$m$ symmetric tensor defined as follows. The $(i_1, \ldots, i_m)$th entry of $H_m(y)$ is given by $\sigma^m h_\alpha(\sigma^{-1}y_1, \dots, \sigma^{-1}y_L)$, where $\alpha \in \{0, \ldots, m\}^L$ denotes the multi-index associated to $(i_1, \ldots, i_m)$: $\alpha_l=|\{j\in [m]\,: i_j=\ell\}|, l \in \{0,\ldots, m\}$. 
Property~3 of the Hermite polynomials implies that if $Y \sim \cN(\mu, \sigma^2 I)$, then $\E[H_m(Y)] = \mu^{\otimes m}$.

Fix $k \geq 1$, and consider the degree-$k$ polynomial
\begin{equation*}
t(y) = \sum_{m=1}^{k} \frac{\langle \Delta_m, H_m(y)\rangle}{(\sqrt 3 \sigma)^{2m} m!}\,.
\end{equation*}
Note that, if $Y \sim \dist{P}_\zeta$, then
\begin{equation*}
\E[t(Y)] = \E\left[\sum_{m=1}^{k} \frac{\langle \Delta_m, \E[H_m(Y) | G]\rangle}{(\sqrt 3 \sigma)^{2m} m!}\right] = \sum_{m=1}^{k} \frac{\langle \Delta_m, \E[(G\zeta)^{\otimes m}]\rangle}{(\sqrt 3 \sigma)^{2m} m!}\,.
\end{equation*}
Thus if $Y \sim \dist{P}_\theta$ and $Y'\sim \dist{P}_\phi$, we get
\begin{equation*}
\E[t(Y)] - \E[t(Y')] = \sum_{m=1}^{k} \frac{\|\Delta_m\|^2}{(\sqrt 3 \sigma)^{2m} m!} =: \delta\,.
\end{equation*}
Lemma~\ref{lem:tensor_bound} in Appendix~\ref{supp} implies that $\delta \leq 4$.
Moreover, by Lemma~\ref{lem:hermite_variance_bound} in Appendix~\ref{supp}, the variances of both $t(Y)$ and $t(Y')$ are bounded above by $e\delta$.
Applying Lemma~\ref{lem:testing_bound} in Appendix~\ref{supp} therefore yields
\begin{equation*}
D(\theta \, \| \, \phi) \geq \frac{\delta^2}{4 e \delta + \delta^2} \geq \frac{\delta^2}{(4e+4)\delta} > \frac{1}{15} \sum_{m=1}^{k} \frac{\|\Delta_m\|^2}{(\sqrt 3 \sigma)^{2m} m!}\,.
\end{equation*}
Since $k \geq 1$ was arbitrary and the summands are nonnegative, letting $k \to \infty$ yields the claim.
\end{proof}

\section{Supplemental materials}\label{supp}
\paragraph{Note:} \emph{In the following sections, we use $C$ and $c$ to represent constants whose value may change from expression to expression and which may depend on $L$ unless otherwise noted.}

\subsection{Comparison of the phase shift model with~\cite{BanChaSin14}}
The phase shift model we propose in Section~\ref{sec:mra} is designed to address some drawbacks of the discrete model for MRA proposed by~\cite{BanChaSin14}. As we show in this section, that model possesses several statistically undesirable properties---namely that the minimax rate of estimation over $\cT_s$ for any $1 \leq s \leq \lfloor L/2 \rfloor$ is very poor, but for reasons that do not shed any light on the statistical difficulties of applications such as cryo-EM.

We first review the discrete MRA model~\cite{BanChaSin14}. Let $\C$ be the group acting on $\RR^L$ by circular shifts of the coordinates. (Note that this group is isomorphic to the cyclic group $\ZZ/L$.) In this model, we observe independent copies of
\begin{equation*}
Y = G\theta + \sigma \xi\,,
\end{equation*}
where $G$ is drawn uniformly at random from $\C$ and $\xi \sim \cN(0, I_L)$ is independent of $G$.
As we note in Section~\ref{sec:mra}, elements of $\C$ can also be viewed as phase shifts in the Fourier domain by $L$th roots of unity.

This section establishes the following lower bound for discrete MRA.
\begin{appxthm}\label{thm:discrete_mra}
Let $1 \leq s \leq \lfloor L/2 \rfloor$.
Let $\cT_s$ be the set of vectors $\theta \in \cT$ satisfying Assumption~\ref{assumption} and $\supp(\hat \theta) \subset [s]$. In the discrete MRA model, for any $\sigma \geq \max_{\theta \in \cT_s} \|\theta\|$,
\begin{equation*}
\inf_{T_n} \sup_{\theta \in \cT_s} \E_\theta[\rho(T_n, \theta)] \geq \frac{C}{L} \Big(\frac{\sigma^L}{\sqrt n} \wedge 1 \Big)\,,
\end{equation*}
where the infimum is taken over all estimators $T_n$ of $\theta$ and where $C$ is a universal constant.
\end{appxthm}

The result follows from the following variant of Proposition~\ref{prop:generalized_matching}, which allows us to exhibit vectors in $\cT_s$ for any $1 \leq s \leq \lfloor L/2 \rfloor$ whose first $L-1$ moments match. The proof then follows from standard minimax technique combined with Theorem~\ref{thm:tight_KL_bound}, as in the proof of Theorem~\ref{thm:main_lower}.

As Proposition~\ref{prop:discrete_matching} makes clear, the pairs of signals which are hard to distinguish under the discrete multi-reference alignment model are ``pure harmonics'' that is, vectors whose Fourier transform is supported only on a single pair of coordinates. An example of two such signals $\theta, \phi \in \R^{17}$ appear in Figure~\ref{fig:discrete_MRA}. These two signals do not lie in the same orbit of $\C$ since there is no phase shift by a $L$th root of unity which makes them coincide, so that $\metric_{\C}(\theta, \phi) \neq 0$. However, it is clear from the perspective of the practitioner that one signal should indeed be viewed as a shift of the other, since they are both discretizations of the same underlying continuous signal. We therefore argue that minimax rates of estimation for the discrete multi-reference alignment model---which are governed by signals like the ones appearing in Figure~\ref{fig:discrete_MRA}---do not accurately reflect the statistical difficulty of problems like cryo-EM in practice.

\begin{figure}[h]
\begin{center}
\includegraphics[width=.5\textwidth]{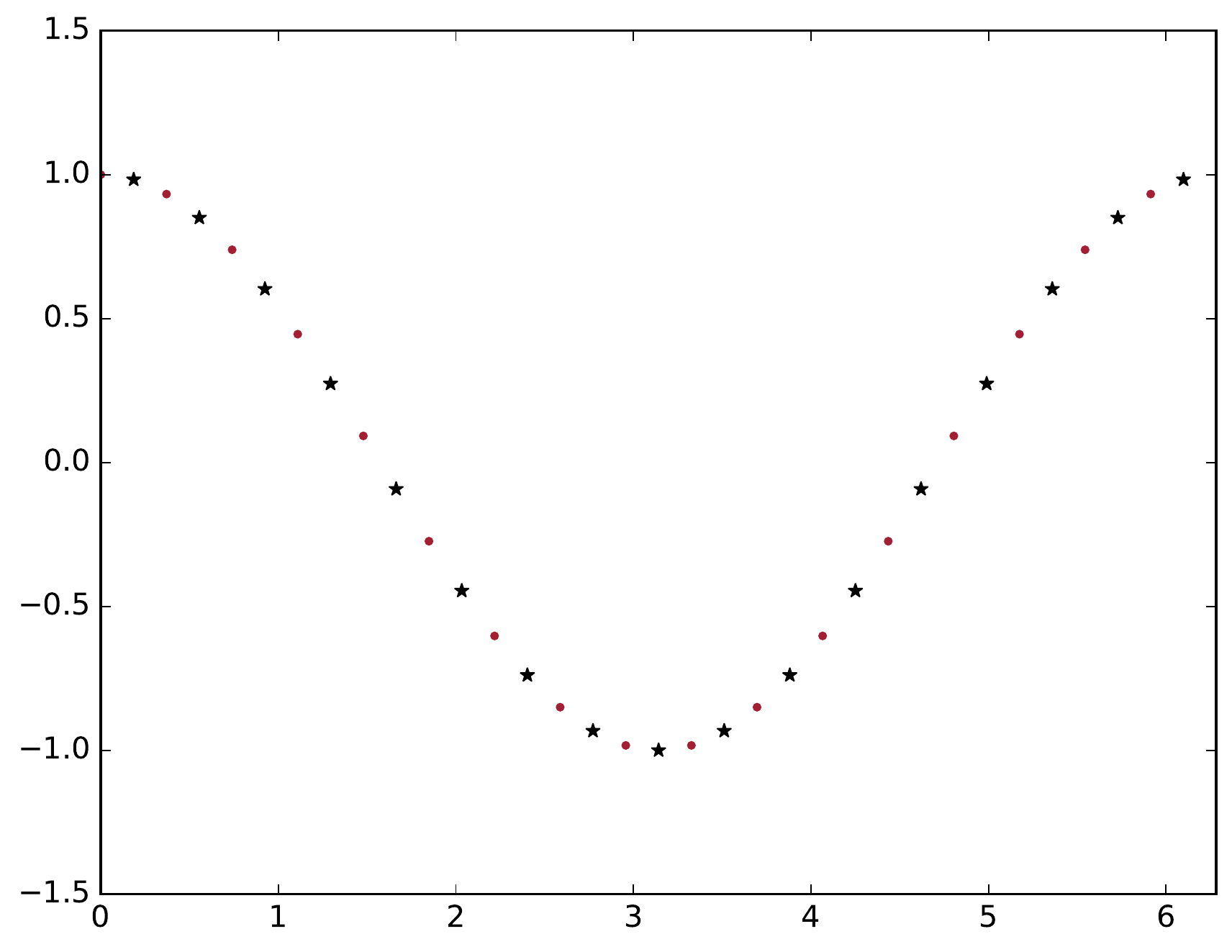}
\caption{Two signals $\theta, \phi \in \R^{17}$, plotted coordinate-wise. The signal $\theta$ (maroon dots) and the signal $\phi$ (black stars) do not lie in the same orbit of $\C$; however, they do lie in the same orbit of $\cS$.}
\label{fig:discrete_MRA}
\end{center}
\end{figure}

\begin{appxprop}\label{prop:discrete_matching}
Suppose $\theta, \phi \in \RR^L$ satisfy
\begin{equation*}
|\hat \theta_{j}| = |\hat \phi_{j}| \quad \text{for $j \in \{\pm 1\}$}
\end{equation*}
and
\begin{equation*}
\hat \theta_j = \hat \phi_j = 0 \quad \text{for $j \notin \{\pm 1\}$}\,.
\end{equation*}
If $G$ is drawn uniformly from $\C$, then for any $m = 1, \dots, L-1$, the $m$th moment tensors satsify
\begin{equation*}
\E[(G \theta)^{\otimes m}] = \E[(G \phi)^{\otimes m}]\,.
\end{equation*}
\end{appxprop}
\begin{proof}
Fix $m < L$.
We follow the proof of Proposition~\ref{prop:generalized_matching}.
It is not hard to verify that in the discrete MRA model, for any $u, \zeta \in \RR^L$,
\begin{equation*}
\E_g[(u^\top G \zeta)^m] = \sum_{j_1 + \dots + j_m \equiv 0} \prod_{n=1}^m \hat u_{-j_n} \hat \zeta_{j_n}\,.
\end{equation*}
Where the sum is taken over all $j_1, \dots, j_m \in [L]$ whose sum is congruent to $0$ modulo $L$. (Note the difference with~\eqref{EQ:sumofji}.)

If $\zeta \in \{\theta, \phi\}$, then $\hat \zeta_j = 0$ for $j \notin \{\pm 1\}$, so it suffices to consider sums whose entires are all $\pm 1$.
Because $m < L$, the only way such a sum can be congruent to $0$ modulo $L$ is if $+1$ and $-1$ occur an equal number of times.
As in the proof of Proposition~\ref{prop:generalized_matching}, this implies that the quantity $\E[(u^\top G \zeta)^m]$ depends only on the modulus $|\hat \zeta_1|$ which is constant over $\zeta \in \{\theta, \phi\}$, which concludes the proof.
\end{proof}

\begin{proof}[Proof of Theorem~\ref{thm:discrete_mra}]

Fix $z = e^{\mathrm{i} \delta}$ for $\delta = c_1 (\sigma^{L}/\sqrt n \wedge 1)$, for some small constant $c_1 < \pi/L$ to be chosen later.
Define $\phi$ by
\begin{equation*}
\hat \phi_j = \left\{\begin{array}{ll}
1 & \text{ if $j \in \{\pm1\}$} \\
0 & \text{ otherwise}\end{array}\right.
\end{equation*}
and $\theta_n$ by
\begin{equation*}
(\widehat{\theta_n})_j = \left\{\begin{array}{ll}
z & \text{ if $j = 1$} \\
z^* & \text{ if $j = -1$} \\
0 & \text{ otherwise.}\end{array}\right.
\end{equation*}
Write $\omega = e^{\mathrm{i}2\pi/L}$ for a primitive $L$th root of unity. The definition of the group $\C$ implies that as long as $\delta < \pi/L$,
\begin{equation*}
\metric^2(\theta_n, \phi) = \min_{G \in \C} \|\theta_n - G\phi\|^2 = 2 \min_{k \in [L]} \|z - \omega^k\|^2 = 2 \|z - 1\|^2\,,
\end{equation*}
so $c \delta^2 \leq \metric^2(\theta_n, \phi) \leq C \delta^2$.
Proposition~\ref{prop:discrete_matching} and Theorem~\ref{thm:tight_KL_bound} imply that, for $c_1$ sufficiently small,
\begin{equation*}
D(\dist{P}_{\theta_n}^n \parallel \dist{P}_\phi^n) = n D(\theta_n \parallel \phi) \leq \overline{C} n \sigma^{-2L} \metric^2(\theta_n, \phi) \leq \frac 12\,,
\end{equation*}
and standard minimax techniques~\cite{Tsy09} yield the bound.
\end{proof}

We note that, by following the proof of Proposition~\ref{prop:lambda_P_bound}, one can show that the optimal $\sigma^L/\sqrt n$ rate is achieved by a modified MLE.
\subsection{Proof of Theorem~\ref{thm:small_s}}
In all proofs, we use $C$ and $c$ to represent constants whose value may change from expression to expression and which may depend on $L$ unless otherwise noted.
We consider the cases $s = 0$ and $s = 1$ separately.

If $s = 0$, then the set $\cT_s$ consists of vectors whose Fourier transforms are supported only on the $0$th coordinate.
In other words, any vector in $\cT_s$ is a multiple of the all-ones vector $\bone$ of $\R^L$.%
The group $\cS$ acts as the identity on this subspace, so for any $\theta \in \cT_s$, the distribution $\dist{P}_\theta$ is just $\cN(\theta, \sigma^2 I_L)$, where all the coordinates of $\theta$ are equal: $\theta = \overline \theta \bone$ for $\overline \theta \in \R$.
We therefore interpret $n$ samples $Y_1, \dots, Y_n$ from $\dist{P}_\theta$ as $nL$ samples $X_1, \dots, X_{nL}$ from $\cN(\bar \theta, \sigma^2)$.%
The estimator $\check \theta_n = \left(\frac{1}{nL} \sum_{i=1}^{nL} X_i\right) \bone$ achieves
\begin{equation*}
\sup_{\theta \in \cT_0} \E_\theta[\metric(\check \theta_n, \theta)] \leq \frac{\sigma}{\sqrt n}\,.
\end{equation*}

If $s = 1$, then $\cT_s$ consists of vectors whose Fourier transforms have support lying in $\{-1, 0, 1\}$.
This subspace is spanned by the orthogonal vectors %
$\bone$,
$u := \{2L^{-1/2} \cos(2 \pi k/L)\}_{k=1}^L$, and $v := \{2L^{-1/2} \sin(2 \pi k/L)\}_{k=1}^L$. Restricted to this space, the action of $\cS$ is isomorphic to the action of $\mathrm{SO}(2)$ given by fixing $\bone$ and rotating the space spanned by $\{u, v\}$.
In particular, this implies that any element of $\cT_1$ is in the same orbit as a vector $\alpha \1 + \beta u$, where $\beta \geq 0$.
We therefore assume without loss of generality that $\theta \in \cT_1$ is of this form.
We write in what follows $a_+$ for the quantity $\max\{a, 0\}$.

We exhibit an estimator achieving
\begin{equation*}
\sup_{\theta \in \cT_1}\E_\theta[\metric(\check \theta_n, \theta)] \leq C \frac{\sigma^2}{\sqrt n}\,.
\end{equation*}
We can assume that $n \geq C_0 \sigma^4$ for some constant $C_0$ to be specified, since otherwise the bound is vacuous.

Recall that, by Assumption~\ref{assumption}, there exists a constant $c_0$ such that if $\theta \in \cT_1$, then either $\hat \theta_1 = 0$ or $|\hat \theta_1| \geq c_0$.
Given samples $Y_1, \dots, Y_n$, as in Lemma~\ref{lem:projection_construction}, define $M_1 = \frac 1 n \sum_{i=1}^n |\widehat{(Y_i)}_1|^2 - \sigma^2$.
The proof of that lemma establishes that $\E[M_1] = |\hat \theta_1|^2$, $\var[M_1] \leq C \sigma^4/n$ and that $M_1$ satisfies a tail bound of the form
\begin{equation*}
\p[|M_1 - \E M_1| \geq \frac 12 c_0^2] \leq 2 \exp(- c n \sigma^{-4})\,,
\end{equation*}
where $c$ is a constant depending on $c_0$.
We denote by $\cE$ the high probability event on which $|M_1 - \E M_1| \leq c_0^2/2$.

Define the thresholded quantity
\begin{equation*}
\tilde M_1 = \left\{\begin{array}{ll}
M_1 & \text{ if $M_1 \geq \frac 12 c_0^2$,} \\
0 & \text{ otherwise}.
\end{array}\right.
\end{equation*}
We will employ $\sqrt{\tilde M_1}$ as an estimator for the quantity $|\hat \theta_1|$, and consequently define
\begin{equation*}
\check \theta_n = \Big(\frac{1}{nL} \sum_{i=1}^n \sum_{k = 1}^L (Y_i)_k\Big) \bone + \sqrt{\tilde M_1} u\,.
\end{equation*}

For $\theta \in \cT_1$, we obtain
\begin{equation*}
\E[\metric^2(\check \theta_n, \theta)] = L\E\Big[\Big(\frac{1}{nL} \sum_{i=1}^n \sum_{k = 1}^L (Y_i)_k - \overline \theta\Big)^2\Big] + 2 \E[(\sqrt{\tilde M_1} - |\hat \theta_1|)^2]\,.
\end{equation*}
As argued above, the first term is at most $\sigma^2/n$. To control the second term, we split our analysis into two cases.

First, suppose $\hat \theta_1 = 0$.
Then the Cauchy-Schwarz inequality implies
\begin{equation*}
\E[(\sqrt{\tilde M_1} - |\hat \theta_1|)^2] = \E [M_1 \1_{M_1 \geq \frac 12 c_0^2}] \leq \E [M_1 \1_{\cE^c}] \leq \Big(\E M_1^2 \p[\cE^c]\Big)^{1/2} \leq C\frac{\sigma^4}{n}\,,
\end{equation*}
which implies the claim when $\hat \theta_1 = 0$.

We now suppose that $|\hat \theta_1| \neq 0$, which implies by Assumptions~\ref{assumption:norm} and~\ref{assumption} that $c_0 \leq |\hat \theta_1| \leq c$.
Note that
\begin{equation*}
\E[(\sqrt{\tilde M_1} - |\hat \theta_1|)^2] = \E[(\sqrt{\tilde M_1} - |\hat \theta_1|)^2\1_{M_1 \geq \frac 12 c_0^2}] + \E[(\sqrt{\tilde M_1} - |\hat \theta_1|)^2\1_{M_1 < \frac 12 c_0^2}]\,.
\end{equation*}
If $M_1 < \frac 12 c_0^2$, then $\tilde M_1 = 0$ and $(\sqrt{\tilde M_1} - |\hat \theta_1|)^2 = |\hat \theta_1|^2 \leq c^2$, so
\begin{equation}\label{eqn:m1_small}
\E[(\sqrt{\tilde M_1} - |\hat \theta_1|)^2\1_{M_1 < \frac 12 c_0^2}] \leq c^2 \p[M_1 < \frac 12 c_0^2] \leq c^2 \p[\cE^c] \leq C \frac{\sigma^4}{n}\,.
\end{equation}

If $M_1 \geq \frac 12 c_0^2$, then $M_1 = \tilde M_1$.
We obtain
\begin{align*}
\E[(M_1 - |\hat \theta_1|^2) \1_{\cE}] & = \E[(\tilde M_1 - |\hat \theta_1|^2) \1_{\cE}] \\
& = \E[(\sqrt{\tilde M_1} - |\hat \theta_1|)^2\1_{\cE}(\sqrt{\tilde M_1} + |\hat \theta_1|)^2] \\
& \geq c_0^2 \E[(\sqrt{\tilde M_1} - |\hat \theta_1|)^2\1_{\cE}]
\end{align*}

Note that
\begin{equation*}
\E[(M_1 - |\hat \theta_1|^2) \1_{\cE}] \leq \E[(M_1 - |\hat \theta_1|^2)] \leq C \frac{\sigma^4}{n}\,.
\end{equation*}
Combining the above displays yields
\begin{equation}\label{eqn:m1_big}
\E[(\sqrt{\tilde M_1} - |\hat \theta_1|)^2\1_{\cE}] \leq c_0^{-2} \E[(M_1 - |\hat \theta_1|^2) \1_{\cE}] \leq C \frac{\sigma^4}{n}\,.
\end{equation}
Together~\eqref{eqn:m1_small} and~\eqref{eqn:m1_big} imply the claimed bound when $\hat \theta_1 \neq 0$.
\qed

\subsection{Proof of proposition~\ref{prop:lambda_P_bound}}

In what follows, the symbols $c$ and $C$ will refer to unspecified positive constants whose value may change from line to line.

By Lemma~\ref{lem:first_moment},
\begin{equation}\label{eqn:k1_bound}
D(\phi) = \frac{1}{2\sigma^2}\| \E[G \theta - G \phi]\|^2 + D(\vartheta \parallel \varphi)\,,
\end{equation}
where $\vartheta = \theta - \E G\theta$ and $\varphi = \phi - \E G \phi$.

If $|\hat \theta_0 - \hat \phi_0| \geq \frac 1 2 \metric(\theta, \phi)$, then~\eqref{eqn:k1_bound} implies
\begin{equation*}
D(\phi) \geq \frac{1}{2\sigma^2}\| \E[G \theta - G \phi]\|^2 = \frac{1}{2\sigma^2}(\hat \theta_0 - \hat \phi_0)^2 \ge \frac{\metric^2(\theta, \phi)}{8\sigma^2} \geq \frac 1 8 \sigma^{-4s+2} \metric^2(\theta, \phi)\,.
\end{equation*}
On the other hand, if $|\hat \theta_0 - \hat \phi_0| < \frac 1 2 \metric(\theta, \phi)$, then
$$
\metric(\vartheta, \varphi)^2 = \metric(\theta_S, \phi)^2  - |\hat \theta_0 - \hat \phi_0|^2 \ge 3\metric^2(\theta, \phi)/4\,.
$$
Thus, by~\eqref{eqn:k1_bound}, it suffices to show that
\begin{equation*}
D(\vartheta \parallel \varphi) \ge C \sigma^{-4s+2} \metric(\vartheta, \varphi)^2\,,
\end{equation*}
for vectors $\vartheta$ and $\varphi$ satisfying $\E G \vartheta  = \E G \varphi = 0$.
In what follows, write $\metric(\vartheta, \varphi) = \ep$.
Since $D(\varphi)=D(G\varphi)$ for all $G \in \cS$, we may assume that $\|\vartheta - \varphi\| = \ep$.
We will show that there exists a small positive constant $c$ such that for some $m \leq 2 s - 1$,
\begin{equation*}
\|\Delta_m\| := \|\E[(G\theta)^{\otimes m} - (G\phi)^{\otimes m}]\| \geq c \ep\,,
\end{equation*}
and the claim will follow from Theorem~\ref{thm:tight_KL_bound}.
We denote by $\ep_0$ and $\kappa$ small constants whose values will be specified.
We split the proof into cases.
\paragraph{Case 1: $\metric(\vartheta, \varphi) \leq \ep_0$}
There are two cases: either $\vartheta$ and $\varphi$ have essentially the same power spectrum (i.e., $|\hat \vartheta_k| \approx |\hat \varphi_k|$ for all $k$) or their power spectra are very different.
We will treat these two cases separately.

Recall that for each $j \in S$, by Assumptions~\ref{assumption:norm} and~\ref{assumption}, the bounds $c_0^{-1}\le| \hat \vartheta_j| \le c$ hold. Consider the polar form $\hat \varphi_j/\hat \vartheta_j = r_j e^{\mathrm{i} \delta_j}$,
where $r_j\ge 0$.

\paragraph{Case 1(a): There exists $j \in S$ such that $|1 - r_j| \ge \kappa \eps$}
The fact that $|\hat \vartheta_j| \geq c_0^{-1}$ implies
\begin{align*}
\|\Delta_2\|^2 & = \|\E[(G \vartheta)^{\otimes 2} - (G \varphi)^{\otimes 2}]\|^2 \\
& = \sum_{k = -\lfloor L/2 \rfloor}^{\lfloor L/2 \rfloor} (|\hat \vartheta_k|^2 - |\hat \varphi_k|^2)^2 \\
& \geq (|\hat \vartheta_j|^2 - |\varphi_j|^2)^2 \\
& \geq |\hat \vartheta_j|^4 (1 - r_j^2)^2 \\
& \geq c_0^{-4} (1+r_j)^2(1-r_j)^2 \\
& \geq c_0^{-4} \kappa^2 \eps^2\,,
\end{align*}
so that $\|\Delta_2\| \geq c \eps$.

\paragraph{Case 1(b): $|1 - r_j| < \kappa \ep$ for all $j \in S$}
Denote by $p$ the smallest integer in $S$ and observe that 
\begin{equation*}
\eps^2= \metric(\vartheta, \varphi)^2 = \min_{z: |z| = 1} 2 \sum_{j \in S} |1-r_jz^je^{\mathrm{i}\delta_j}|^2|\hat \vartheta_j|^2 \le C \sum_{j \in S} |1 - r_j e^{\mathrm{i}(p\delta_j-j \delta_p)/p}|^2\,,
\end{equation*}
where the inequality follows from choosing $z = e^{-\mathrm{i} \delta_p/p}$.
Therefore, there exists a coordinate $\ell \in S$ such that
\begin{equation}\label{EQ:deltaphases}
|1 - e^{\mathrm{i} (p \delta_\ell - \ell \delta_p)/p}| \geq |1 - r_\ell e^{\mathrm{i}(p \delta_\ell - \ell \delta_p)/p}| - \kappa \ep \geq c \ep\,,
\end{equation}
as long as $\kappa $ is chosen sufficiently small. In particular, $|1 - e^{\mathrm{i} (p \delta_\ell - \ell \delta_p)/p}| > 0$, so $\ell \neq p$.
Note that this fact implies that, if $|1 - r_j| < \kappa \ep$ for all $j \in S$, then $|S| \geq 2$.

Choose $m = \ell + p$.
Since $\ell, p \in S \subseteq [s]$ and $\ell \neq p$, the bound $m \leq 2s - 1$ holds.
As in the proof of Proposition~\ref{prop:generalized_matching}, we have that
\begin{align*}
\|\Delta_m\|^2 & = \sum_{j_1 + \dots + j_{m} = 0} \left|\prod_{n = 1}^{m} \hat \vartheta_{j_n}- \prod_{n = 1}^{m} \hat \varphi_{j_n}\right|^2 \\
& = \sum_{j_1 + \dots + j_{m} = 0} \left|1 - \prod_{n = 1}^{m} r_{j_n} e^{\mathrm{i} \delta_{j_n}}\right|^2 \prod_{n = 1}^{m} |\hat \vartheta_{j_n}|^2\,.
\end{align*}
Each term in the above sum is positive.
One valid solution to the equation $j_1 + \dots + j_m = 0$ is $j_1 = \dots = j_\ell = -p$ and $j_{\ell + 1} = \dots = j_m = \ell$.
We obtain
\begin{align*}
\|\Delta_m\|^2 &\ge C \left|1 - e^{i (p\delta_\ell - \ell \delta_p)} \prod_{n = 1}^{m} r_{j_n} \right|^2 \\
& \geq C |1 - e^{i (p\delta_\ell - \ell \delta_p)}|^2 - C\left|1 - \prod_{n = 1}^{m} r_{j_n}\right|^2\,.
\end{align*}
As long as $\kappa \ep$ is small enough, $\left|1 - \prod_{n = 1}^{m} r_{j_n}\right| \leq 2 m \kappa \ep$. Moreover, as long as $\ep_0 $ is chosen sufficiently small, $\delta_\ell$ and $\delta_p$ can both be chosen small enough that $|p \delta_\ell - \ell \delta_p| \leq 1$, in which case it holds
$$
|1 - e^{i (p\delta_\ell - \ell \delta_p)}|^2 \ge |1 - e^{i (p \delta_\ell - \ell \delta_p)/p}|^2 \ge c^2\eps^2\,,
$$
where the last inequality follows from~\eqref{EQ:deltaphases}.
So for $\ep_0$ and $\kappa$ chosen sufficiently small, this proves the existence of an $m \leq 2s -1$ for which $\|\Delta_m\| \geq c \ep$.

\paragraph{Case 2: $\metric(\vartheta, \varphi) > \eps_0$ and $\|\varphi\| > 3\|\vartheta\|$}
Lemma~\ref{lem:second_moment_bound} implies that
\begin{equation*}
\|\Delta_2\|^2 \geq \frac{1}{4L} \ep^4 \geq \frac{\ep_0^2}{4L} \ep^2\,,
\end{equation*}
so there exists a constant $c$ such that $\|\Delta_2\| \geq c \ep$.

\paragraph{Case 3: $\metric(\vartheta, \varphi) > \eps_0$ and $\|\varphi\| \leq 3\|\vartheta\|$}
For positive integers $j, k$, denote by $[j, k]$ their greatest common divisor.
Given a vector $\zeta \in \R^L$, denote by $\cP$ the following set of polynomials in the entries of $\hat \zeta$:
\begin{align*}
p_0(\zeta) & = \hat \zeta_0, \\
p_j(\zeta) & = |\hat \zeta_j|^2 \quad \text{for $1 \leq j \leq \lfloor L/2 \rfloor$,} \\
p_{jk}(\zeta) & = \hat \zeta_{-k}^{j/[j, k]} \hat \zeta_j^{k/[j, k]} \text{for $1 \leq j,k \leq \lfloor s \rfloor$.}
\end{align*}
Note that each of the the polynomials in $\cP$ appear as entries in $\E[(G\zeta)^{\otimes m}]$ for some $m \leq 2s -1$.

If $\metric(\vartheta, \varphi) > 0$, then by~\cite{KakIve93}, there exists at least one polynomial $p \in \cP$ such that
\begin{equation*}
p(\vartheta) \neq p(\varphi)\,.
\end{equation*}
For all $\vartheta \in \R^L$, define $B_{\vartheta, \ep_0} = \{\varphi: \rho(\vartheta, \varphi) \geq \ep_0, \|\varphi\| \leq 3\|\vartheta\|\}$.
It is clear that $B_{\vartheta, \ep_0}$ is compact so that
\begin{equation*}
\delta = \inf_{\vartheta \in \cT_s} \inf_{\varphi \in B_{\vartheta, \ep_0}} \min_{p \in \cP: p(\varphi) \neq p(\vartheta)} |p(\varphi) - p(\vartheta)|>0\,.
\end{equation*}
Note that $\delta$ does not depend on $\theta$ or $\phi$.
Since $\rho(\vartheta, \varphi) > \ep_0$ by assumption, there exists a $p \in \cP$ such that
$|p(\varphi) - p(\vartheta)| \geq \delta.$
Therefore there exists a positive integer $m \leq 2s -1$ such that $\|\Delta_m\| \geq \delta$, and since $\metric(\vartheta, \varphi) \leq \|\varphi\| + \|\vartheta\| \leq C$ for $\varphi \in B_{\vartheta, \ep_0}$, we obtain
\begin{equation*}
\|\Delta_m\| \geq \delta \geq c \ep\,.
\end{equation*}

\subsection{Additional lemmas}

\begin{appxlem}\label{lem:logsumexp_lipschitz}
Let $G$ be a random element drawn according to the Haar probability measure on any compact subgroup $\cG$ of the orthogonal group in $L$ dimensions.
For any $u \in \RR^L$, the function
\begin{equation*}
g(x) = \log \E \exp(u^\top G x)
\end{equation*}
is $\|u\|$-Lipschitz with respect to the Euclidean distance on $\R^L$.
\end{appxlem}
\begin{proof}
Differentiating $g$ yields
\begin{align*}
\|\nabla g(x)\| & = \left\|\frac{ \E \, G^\top u \exp(u^\top G x)}{\E \exp(u^\top G x)}\right\| \\
& \leq \frac{\E \|G^\top u\| \exp(u^\top G x)}{ \E \exp(u^\top G x)} \\
& = \|u\|\,,
\end{align*}
which implies the claim.
\end{proof}

\begin{appxlem}\label{lem:symmetrized_subgaussian}
If $X$ is subgaussian with variance proxy $\sigma^2$ and $\ep$ is a Rademacher random variable independent of $X$, then $\ep X$ is subgaussian with variance proxy $\sigma^2 + (\E X)^2$.
\end{appxlem}
\begin{proof}
We aim to show that if $X$ satisfies
\begin{equation*}
\E \exp( t (X - \E X)) \leq \exp(t^2 \sigma^2/2) \quad \quad \forall t \in \R\,,
\end{equation*}
and if $\ep$ is a Rademacher random variable independent of $X$, then
\begin{equation*}
\E \exp( t (\ep X - \E \ep X)) = \E \exp( t \ep X ) \leq \exp(t^2 (\sigma^2 + (\E X)^2)/2) \quad \quad \forall t \in \R\,.
\end{equation*}
Conditioning on $\ep$ yields
\begin{align*}
\E \exp( t \ep X ) & = \E[\E[\exp( t \ep X ) \mid \ep]] \\
& = \E[ \E[\exp(t \ep (X - \E X)) \mid \ep] \exp (t \ep \E X)] \\
& \leq \exp(t^2 \sigma^2/2)  \E[ \exp (t \ep \E X) ] \\
& \leq \exp(t^2 (\sigma^2 + (\E X)^2)/2)\,,
\end{align*}
where the last step uses Hoeffding's lemma.
This proves the claim.
\end{proof}

\begin{appxlem}\label{lem:dominance}
If $X$ and $Y$ are random variables satisfying $|X| \leq |Y|$ almost surely, and if $\ep$ is a Rademacher random variable independent of $X$ and $Y$, then
\begin{equation*}
\E \exp(t \ep X) \leq \E \exp(t \ep Y) \quad \quad \forall t \in \R\,.
\end{equation*}
\end{appxlem}
\begin{proof}
The function $x \mapsto \cosh(x)$ is increasing on $[0, \infty)$, so
\begin{equation*}
\E[\exp(t \ep X)\mid X] = \cosh(|tX|) \leq \cosh(|tY|) = \E[\exp(t \ep Y) \mid Y]
\end{equation*}
almost surely. The claim follows.
\end{proof}

\begin{appxlem}\label{lem:bdHessian}
Let $H(\zeta)$ and $H_n(\zeta)$ be the Hessians of $D(\phi)$ and $D_n(\phi)$, respectively, evaluated at $\phi = \zeta$.
If $\cB_\eps:=\{\phi \in \R^L\,:\, \rho(\phi, \theta) \le \eps\}$, then
\begin{equation*}
\E\sup_{\phi \in \cB_\eps} \|H(\phi) - H_n(\phi)\|_\mathrm{op}^2 \leq C \frac{\log n}{n \sigma^4}\,.
\end{equation*}
\end{appxlem}
\begin{proof}
The matrix $H_n(\phi)$ can be written as a sum of independent random matrices:
\begin{equation*}
H_n(\phi) = \frac 1 n \sum_{i=1}^n J_i(\phi)\,,\qquad J_i(\phi)=\nabla_\phi^2 \log \frac{f_\theta}{f_\phi}(Y_i)\,.%
\end{equation*}
Using symmetrization, we get
\begin{equation}
\label{EQ:symHess}
\E\sup_{\phi \in \cB_\eps} \|H(\phi) - H_n(\phi)\|_\text{op}^2 \le \frac{4}{n^2} \E\sup_{\phi \in \cB_\eps} \big\|\sum_{i=1}^n \eps_i J_i(\phi)\big\|_\text{op}^2\,,
\end{equation}
where $\eps_1, \ldots, \eps_n$ are i.i.d Rademacher random variables that are independent of $J_1, \ldots, J_n$.

By Lemma~\ref{lem:derivative_bound}, for any $u, \phi, \eta \in \R^L$ such that $\|u\|=1$ and $\phi, \eta \in \cB_\eps$, we have 
$$
|u^\top J_i(\phi) u - u^\top J_i(\eta) u | \le 6 \sigma^{-6}\|Y_i\|^3\|\phi-\eta\| \leq C\frac{1+|\xi_i|^3}{\sigma^3}\|\phi-\eta\|
$$
where $\xi_i$ is Gaussian noise.

Fix $\gamma \in (0,\eps)$ and let $\cZ$ be a $\gamma$-net of $\cB_\eps$.
In other words, we require that
\begin{equation*}
\max_{\phi \in \cB_\eps} \min_{\eta \in \cZ} \|\eta - \phi\| \le \gamma\,.
\end{equation*}
We can always choose $\cZ$ to have cardinality $|\cZ|\le  (C/\gamma)^L$ for some universal constant $C>0$.
Then, by Young's inequality, we get
\begin{equation}
\label{EQ:supnet1}
\sup_{\phi \in \cB_\eps} \big\|\sum_{i=1}^n \eps_i J_i(\phi)\big\|_\text{op}^2 \le  C\frac{\gamma^2}{\sigma^{6}}\big(\sum_{i=1}^n 1+|\xi_i|^3\big)^2 + C \max_{\phi \in \cZ} \big\|\sum_{i=1}^n \eps_i J_i(\phi)\big\|_\text{op}^2
\end{equation}
The expectation of the first term is controlled using the fact that
\begin{equation}
\label{EQ:supnet2}
\E\big(\sum_{i=1}^n 1+ |\xi_i|^3\big)^2=\sum_{i,j=1}^n \E[(1+|\xi_i|^3)(1+|\xi_j|^3)]\le C n^2\,.
\end{equation}
For second term, we employ a standard matrix concentration bound~\cite[Theorem~4.6.1]{Tro15} to get that
$$
\p\Big[\max_{\phi \in \cZ} \big\|\sum_{i=1}^n \eps_i J_i(\phi)\big\|_\text{op}^2 \geq t \Big| Y_1, \dots, Y_n \Big] \leq 2 L |\cZ| \exp\left(-\frac{t}{2 \max_{\phi \in \cZ} \|\sum_{i=1}^nJ_i(\phi)^2\|_\text{op}}\right)\,,
$$
Integrating this tail bound yields
\begin{align*}
\E\big[\max_{\phi \in \cZ} \big\|\sum_{i=1}^n \eps_i J_i(\phi)\big\|_\text{op}^2\big] &\le C \log (L |\cZ|) \E\big[\max_{\phi \in \cZ}\big\|\sum_{i=1}^n J_i(\phi)^2\big\|_\text{op}\big] \\
&\le   C \log (L |\cZ|) n\E\big[\max_{\phi \in \cZ} \|J_1(\phi)^2\|_\text{op}\big] 
\end{align*}
By Lemma~\ref{lem:derivative_bound},
$$
\|J_1(\phi)^2\big\|_\text{op} = \|J_1(\phi)\big\|_\text{op}^2 \le 2\sigma^{-4} + 8 \sigma^{-8}\|Y_1\|^4 \leq C \frac{1+ \|\xi_1\|^4}{\sigma^4}\,.
$$
The above two displays yield
$$
\E\big[\max_{\phi \in \cZ} \big\|\sum_{i=1}^n \eps_i J_i(\phi)\big\|_\text{op}^2\big] \le C\frac{\log (1/\gamma)}{\sigma^4} n 
$$

Combining the last display with~\eqref{EQ:symHess},~\eqref{EQ:supnet1}, and~\eqref{EQ:supnet2}, we get
\begin{equation}
\E\sup_{\phi \in \cB_\eps} \|H(\phi) - H_n(\phi)\|_\text{op}^2 \le C\Big(\frac{\gamma^2}{\sigma^6} + \frac{\log(\eps/\gamma)}{n\sigma^4} \Big)\le C\frac{\log n}{n\sigma^4}\,,
\end{equation}
for $\gamma=n^{-1/2}$.
\end{proof}
\begin{appxlem}
\label{eventA}
Assume the conditions of Theorem~\ref{thm:non_asymptotic} hold. Then the MLE $\tilde \theta_n$ satisfies
\begin{equation*}
\E[\metric(\tilde \theta, \theta)^2] \leq C\frac{\sigma^{4k-2}}{n}\,.
\end{equation*}
\end{appxlem}
\begin{proof}
As in the proof of Theorem~\ref{thm:non_asymptotic}, since $\theta$ is fixed, we simply write $D(\phi)=D(\theta \|\phi)$  and define 
\begin{equation*}
D_n(\phi) = \frac 1 n \sum_{i=1}^n \log \frac{f_\theta}{f_\phi}(Y_i)\,,
\end{equation*}
where $Y_i$ are i.i.d from $\dist{P}_\theta$ and we recall that $f_\zeta$ is the density of $\dist{P}_\zeta, \zeta \in \R^d$.

We first establish using Lemma~\ref{LEM:subG} that the process $\{\mathfrak{G}_n(\phi)\}_{\phi \in \R^d}$ defined by $\mathfrak{G}_n(\phi) = D(\phi) - D_n(\phi)$ is a subgaussian process with respect to the Euclidean distance with variance proxy $c\sigma^2/n$ for some constant $c>0$, i.e., that for any $\lambda \in \R$, we have
\begin{equation*}
\E[\exp(\lambda(\mathfrak{G}_n(\phi) - \mathfrak{G}_n(\zeta)))] \leq \exp\big(c\frac{\lambda^2\sigma^2 }{2n} \|\phi - \zeta\|^2\big)\,.
\end{equation*}

We then apply the following standard tail bound.
\begin{appxprop}[{\cite[Theorem 8.1.6]{Ver17}}]\label{prop:chaining_tail_bound}
If $\{X_\phi\}_\phi$ is a (standard) subgaussian process on $\R^d$ with respect to the Euclidean metric and $B_\delta(\theta)$ is a ball of radius $\delta$ around $\theta$, then
\begin{equation*}
\p[\sup_{\phi \in B_\delta(\theta)} (X_\phi - X_\theta) \geq C \delta + x] \leq C e^{-Cx^2/\delta^2}\,.
\end{equation*}
\end{appxprop}

The rescaled process $\sigma\sqrt n \mathfrak{G}_n$ is standard subgaussian  process with respect to the Euclidean metric, so applying Proposition~\ref{prop:chaining_tail_bound} and noting that $\mathfrak{G}_n(\theta) =0$ yields
\begin{equation}\label{eqn:divergence_tail_bound}
\p\left[\sup_{\phi \in B_\delta(\theta)} \mathfrak{G}_n(\phi) \geq C \frac{\delta }{\sigma\sqrt n} + x\right] \leq C \exp\big(- C \frac{n \sigma^2 x^2}{\delta^2}\big)\,.
\end{equation}

For convenience, write $v_n = \sqrt{n}(\tilde \theta- \theta)$, where $\tilde \theta$ is a MLE satisfying $\|\tilde \theta - \theta\| = \metric(\tilde \theta, \theta)$.
We wish to show that $\E[\|v_n\|^2] \leq C \sigma^{4k-2}$.

We employ the so-called \emph{slicing} (a.k.a \emph{peeling}) method. Define the sequence $\{\alpha_j\}_{j \ge 0}$ where $\alpha_0 = 0$ and $\alpha_j=C_0\sigma^{2k-1}2^j$ for $j \geq 1$ for some constant $C_0$ to be specified.
For any $j \ge 0$, define $S_{j} = \{ \phi \in \R^d\,:\, \alpha_j \leq \sqrt n \metric(\phi, \theta) \leq \alpha_{j + 1}\}$.
We obtain
\begin{align}
\E[\|v_n\|^2] & =\sum_{j\ge 0} \E[\|v_n\|^2 \mid \tilde \theta\in S_j] \p[\tilde \theta\in S_{j}]\nonumber\\
 & \le C_0 \sigma^{4k-2} + \sum_{j \ge 1} \alpha_{j+1}^2 \p[\tilde \theta\in S_{j}] \label{EQ:prvn1}\,.
\end{align}

We now show that if $\tilde \theta\in S_{j}, j \ge 1$, then $\mathfrak{G}_n(\tilde \theta)=D(\tilde \theta) - D_n(\tilde \theta)$ is large. To that end, observe that on the one hand, by definition of the MLE, we have $D_n(\tilde \theta)\le D_n(\theta) = 0$. 
On the other hand, $D(\tilde \theta) \geq C \sigma^{-2k} \metric(\tilde \theta, \theta)^2$.
Hence, if $\tilde \theta\in S_{j}$, then $\mathfrak{G}_n(\tilde \theta) \geq C \sigma^{-2k} \metric(\tilde \theta, \theta)^2 \geq C \sigma^{-2k} \alpha_j^2/n$.
It yields
\begin{align*}
\p[\tilde \theta\in S_{j}]& \leq \p\big[\sup_{\phi \in S_{j}} \mathfrak{G}_n(\phi) \geq C\sigma^{-2k} \frac{\alpha_j^2}{n}\big] 
\le\p\big[\sup_{\phi\in B_{\frac{\alpha_{j+1}}{\sqrt n}} (\theta)}\mathfrak{G}_n(\phi) \geq C\sigma^{-2k} \frac{\alpha_j^2}{n}\big]\,.
\end{align*}

Recall $\alpha_j=C_0\sigma^{2k-1}2^j$ so that $\sigma^{-2k}\alpha_j^2 \ge C \alpha_{j+1}/\sigma$ as long as $C_0$ is sufficiently large.
Apply~\eqref{eqn:divergence_tail_bound} with $\delta=\alpha_{j+1}/\sqrt{n}$ and $x=C \sigma^{-2k} \alpha_j^2/n$ to get
\begin{align*}
\p\big[\sup_{\phi\in B_{\frac{\alpha_{j+1}}{\sqrt n}} (\theta)}\mathfrak{G}_n(\phi) \geq C\sigma^{-2k} \frac{\alpha_j^2}{n}\big]&\le C \exp\big(-C\frac{\alpha_{j}^4}{\alpha_{j+1}^2\sigma^{4k-2}}  \big)\\
&\le C \exp\big(-C2^{2j}  \big)\,.
\end{align*}
Together with~\eqref{EQ:prvn1}, we obtain
\begin{equation}
\label{EQ:prvn2}
\E[\|v_n\|^2] \le C_0 \sigma^{4k -2}\sum_{j\ge 0} 2^{2j} \exp\big(-C2^{2j}  \big)\le C \sigma^{4k-2}\,.
\end{equation}
\end{proof}

\begin{appxlem}
\label{LEM:subG}
The process $\{\mathfrak{G}_n(\phi)\}_{\phi \in \R^d}$ defined by $\mathfrak{G}_n(\phi) = D(\phi) - D_n(\phi)$ is a subgaussian process with respect to the $\ell_2$ distance on $\R^L$ with variance proxy $20 L/(n\sigma^2)$, i.e., for any $\lambda \in \R$, we have
\begin{equation*}
\E[\exp(\lambda(\mathfrak{G}_n(\phi) - \mathfrak{G}_n(\zeta)))] \leq \exp\big(\lambda^2 \frac{10 L}{n\sigma^2 } \|\phi - \zeta\|^2\big)\,.
\end{equation*}
\end{appxlem}
\begin{proof}
By definition of $\mathfrak{G}_n$ and the densities $f_\zeta$ and $f_\phi$, we have
\begin{align*}
\mathfrak{G}_n(\phi) - &\mathfrak{G}_n(\zeta)  =  D(\phi) - D(\zeta) - D_n(\phi) + D_n(\zeta) \\
& = \E[\log f_\zeta(Y) - \log f_\phi(Y)] - \frac 1 n \sum_{i=1}^n \log f_\zeta(Y_i) - \log f_\phi(Y_i) \\
& =  \E\Big[\log \frac{\E [\exp(-\frac{ \|Y - G \zeta\|^2}{2\sigma^2})\mid Y]}{\E [\exp(-\frac{ \|Y - G \phi\|^2}{2\sigma^2}) \mid Y]}\Big]  -  \frac 1 n \sum_{i=1}^n \log \frac{\E [\exp(-\frac{ \|Y_i - G \zeta\|^2}{2\sigma^2})\mid Y_i]}{\E [\exp(-\frac{ \|Y_i - G \phi\|^2}{2\sigma^2}) \mid Y_i]} \\
&=\E\Big[\log \frac{\E [ \exp({Y^\top G \zeta}{/\sigma^2}) \mid Y]}{\E [ \exp({Y^\top G \phi}{/\sigma^2}) \mid Y]}\Big] - \frac 1 n \sum_{i=1}^n \log \frac{\E [\exp({Y_i^\top G \zeta}{/\sigma^2}) \mid Y_i]}{\E [ \exp({Y_i^\top G \phi}{/\sigma^2}) \mid Y_i]} \\
&= \E[\Delta(Y)] - \frac 1 n \sum_{i=1}^n \Delta(Y_i)\,,
\end{align*}
where
\begin{equation*}
\Delta(y) = \log \frac{\E \exp({y^\top G \zeta}{/\sigma^2})}{\E \exp({y^\top G \phi}{/\sigma^2})}\,.
\end{equation*}
Next, using a standard symmetrization argument, we get
\begin{equation}
\label{EQ:symm}
\E[\exp(\lambda(\mathfrak G_n(\phi) - \mathfrak G_n(\zeta)))] \le \prod_{i=1}^n\E[\exp(\frac{2\lambda}{n} \ep_i \Delta(Y_i))]\,,
\end{equation}
where  $\eps_1, \ldots, \eps_n$ are i.i.d Rademacher random variables that are independent of $\Delta(Y_1), \ldots, \Delta(Y_n)$.

Next, Lemma~\ref{lem:logsumexp_lipschitz} implies $\zeta \mapsto  \log \E \exp( Y_i^\top G \zeta/\sigma^2)$ is $\|Y_i\|/\sigma^2$-Lipschitz with respect to the Euclidean distance. Hence $|\Delta(Y_i)| \leq \|Y_i\|\|\phi - \zeta\|/\sigma^2 \leq (\|\theta\| + \sigma \|\xi_i\|)\|\phi - \zeta\|/\sigma^2$.
The function $\xi \mapsto (\|\theta\| + \sigma \|\xi\|)$ is $\sigma$-Lipschitz, so Gaussian concentration~\cite[Theorem~5.5]{BouLugMas13} implies that $(\|\theta\| + \sigma \|\xi_i\|)$ is subgaussian with variance proxy $\sigma^2$.
We also have
\begin{equation*}
\E (\|\theta\| + \sigma \|\xi_i\|) \leq \sigma (1 + \sqrt L) \leq 2 \sqrt L \sigma.
\end{equation*}
Applying Lemma~\ref{lem:symmetrized_subgaussian} yields that $\ep_i (\|\theta\| + \sigma \|\xi_i\|)$ is subgaussian with variance proxy $ 5 L \sigma^2$. Combining this fact with Lemma~\ref{lem:dominance}, we obtain
\begin{align*}
\E\big[\exp(\frac{2\lambda}{n} \ep_i \Delta(Y_i))\big] & \leq \E\big[\exp(\frac{2\lambda\|\phi - \zeta\|}{\sigma^2n} \ep_i (\|\theta\| + \sigma \|\xi_i\|))\big] \\
& \leq \exp(\frac{10 L \lambda^2 \|\phi - \zeta\|^2}{\sigma^2 n^2})
\end{align*}
for all $\lambda \in \R$. Together with~\eqref{EQ:symm}, this yields the desired result.

\end{proof}

\begin{appxlem}\label{lem:projection_construction}
Fix $\theta \in \cT$, and assume that $\sigma \geq \|\theta\| \vee 1$.
For any $j =- \lfloor L/2\rfloor, \ldots, \lfloor L/2\rfloor $, define,
$$
M_j = \frac 1 n \sum_{i=1}^n |\widehat{(Y_i)}_j|^2 - \sigma^2\,.
$$
Define the set $\tilde S$ by
$$
\tilde S=\Big\{ j \in \{- \lfloor L/2\rfloor, \ldots, \lfloor L/2\rfloor\}\,:\,  M_j \geq \frac 12 c_0^2  \Big\}\,,
$$
where $c_0$ is a lower bound on the magnitude of nonzero Fourier coefficients of elements of $\cT$. (See Assumption~\ref{assumption}.)
Then there exists a constant $c$ depending on $c_0$ such that
\begin{equation*}
\p[\tilde S \neq \supp(\hat \theta)] \leq 2 L\exp(-c n \sigma^{-4})\,.
\end{equation*}
\end{appxlem}
\begin{proof}
It is straightforward to check that $\E[M_j] = |\hat \theta_k|^2$.
We now establish that the random variable $|\widehat{(Y_i)}_j|^2$ is $O(\sigma^2)$-subexponential, i.e., there exists a positive constant $c_1$ such that
\begin{equation*}
\E \exp(t (|\widehat{(Y_i)}_j|^2 - \E|\widehat{(Y_i)}_j|^2)) \leq \exp(c_1^2 \sigma^4 t^2) \quad \quad \forall |t| \leq \frac{1}{c_1\sigma^2}\,.
\end{equation*}

This follows from the following considerations. It is clear that $|\widehat{(Y_i)}_j| \leq |\widehat{G_i\theta}_j| + \sigma|\widehat{\xi_i}_j| \leq \sigma(1 + |\widehat{\xi_i}_j|)$. Since $\xi \mapsto \sigma(1 + |\widehat{\xi_i}_j|)$ is $\sigma$-Lipschitz, Gaussian concentration~\cite[Theorem~5.5]{BouLugMas13} implies that $\sigma(1 + |\widehat{\xi_i}_j|)$ is subgaussian with variance proxy $\sigma^2$. Since
\begin{equation*}
\E[\sigma(1 + |\widehat{\xi_i}_j|)] \leq  C\sigma\,,
\end{equation*}
we obtain that there exists a constant $c_2$ such that
\begin{equation*}
\E[\exp(t|\widehat{(Y_i)}_j|)] \leq \E[\exp(t(\sigma(1 + |\widehat{\xi_i}_j|)))] \leq \exp(c_2 t^2  \sigma^2)\,.
\end{equation*}
That $|\widehat{(Y_i)}_j|^2$ is $O(\sigma^2)$-subexponential then follows from~\cite[Section~2.7]{Ver17}.
In particular, this implies that $M_1$ has variance at most $C \sigma^4/n$ for some constant $C$ and that $M_1$ satisfies a tail bound of the form
\begin{equation*}
\p[|M_1 - \E M_1| \geq t] \leq 2 \exp(- c (t \wedge t^2) n \sigma^{-4})\,,
\end{equation*}
for some positive constant $c$.

If $j \in \supp(\hat \theta)$, then $|\hat \theta_j| \geq c_0$ by assumption, so
\begin{equation*}
\p[j \notin \tilde S] = \p\big[M_j \leq \frac 12 c_0^2\big] \leq \p\big[|M_j - \E M_j| \geq \frac 12 c_0^2\big] \leq 2 \exp(-c n \sigma^{-4})
\end{equation*}
for some constant $c$ depending on $c_0$. Likewise,
if $j \notin \supp(\hat \theta)$, then $\E[M_j] = 0$ and
\begin{equation*}
\p[j \in \tilde S] \leq \p[M_j \geq \frac 12 c_0^2] \le 2 \exp(-c n \sigma^{-4})\,.
\end{equation*}

The proof follows using a union bound.
\end{proof}

\begin{appxlem}\label{lem:testing_bound}
Let $\dist{P}_0$ and $\dist{P}_1$ be any two distributions on a space $\cX$.
If there exists a measurable function $T: \cX \to \R$ such that $(\E_0[T(X)] - \E_1[T(X)])^2 = \mu^2$ and $\max\{\var_1(T(X)), \var_0(T(X))\} \leq \sigma^2$,
then
\begin{equation*}
D(\dist{P}_0 \| \dist{P}_1) \geq \frac{\mu^2}{4 \sigma^2 + \mu^2}\,.
\end{equation*}
\end{appxlem}
\begin{proof}
We can assume without loss of generality that $\E_0[T(X)] = \mu/2$ and $\E_1[T(X)] = - \mu/2$.
For $i \in \{0, 1\}$, denote by $\dist{Q}_i$ the distribution of $T(X)$ when $X$ is distributed according to $\dist{P}_i$. By the data processing inequality, it suffices to prove the claimed bound for $D(\dist{Q}_0 \| \dist{Q}_1)$.
We can assume that $\dist{Q}_0$ is absolutely continuous with respect to $\dist{Q}_1$ because otherwise the bound is trivial.

Let $f(x) = x \log x - \frac{(x-1)^2}{2(x+1)}$, and note that $f(1) = 0$. Since $f$ is convex on $[0, +\infty)$,
\begin{equation*}
\E_{\dist{Q}_1}\left[f\left(\frac{dQ_0}{dQ_1}\right)\right] \geq f\left(\E_{\dist{Q}_1}\frac{dQ_0}{dQ_1}\right) = f(1) = 0\,.
\end{equation*}
Suppose that $\dist{Q}_1$ and $\dist{Q}_0$ have densities $q_1$ and $q_0$ with respect to some reference measure $\nu$. The preceding calculation implies
\begin{equation*}
D(\dist{Q}_0 \| \dist{Q}_1) = \E_{\dist{Q}_1} \left[ \frac{dQ_0}{dQ_1} \log \frac{dQ_0}{dQ_1} \right] \geq \frac 12 \int \frac{(q_0(x) - q_1(x))^2}{(q_0(x) + q_1(x))} \,\mathrm{d}\nu(x)\,.
\end{equation*}
By the Cauchy-Schwarz inequality,
\begin{align*}
\mu^2 = \left(\int x (q_0(x) - q_1(x)) \,\mathrm{d}\nu(x) \right)^2 & \leq \int x^2 (q_0(x) + q_1(x)) \,\mathrm{d}\nu(x)\int \frac{(q_0(x) - q_1(x))^2}{q_0(x) + q_1(x)} \,\mathrm{d}\nu(x)  \\
& \leq  (2 \sigma^2 + \mu^2/2) \int \frac{(q_0(x) - q_1(x))^2}{q_0(x) + q_1(x)} \,\mathrm{d}\nu(x)\,.
\end{align*}
Therefore
\begin{equation*}
D(\dist{Q}_0 \| \dist{Q}_1) \geq \frac{\mu^2}{4 \sigma^4 + \mu^2}\,,
\end{equation*}
as claimed.
\end{proof}

\begin{appxlem}\label{lem:tensor_bound}
For any $m \geq 1$ and $\theta, \phi \in \RR^L$ satisfying $\|\theta\| = 1$ and $\metric(\theta, \phi) \leq 1/3$,
\begin{equation*}
\|\Delta_m\|^2 = \|\E[(G\theta)^{\otimes m} - (G \phi)^{\otimes m}]\|^2 \leq 12 \cdot 2^{m}\metric^2(\theta, \phi)\,.
\end{equation*}
\end{appxlem}
\begin{proof}
Assume without loss of generality that $\metric(\theta, \phi) = \|\theta - \phi\| =: \ep$.
By Jensen's inequality,
\begin{equation*}
\|\E[(G\theta)^{\otimes m} - (G \phi)^{\otimes m}]\|^2 \leq \E\|(G\theta)^{\otimes m} - (G \phi)^{\otimes m}\|^2 = \|\theta^{\otimes m} - \phi^{\otimes m}\|^2\,.
\end{equation*}
Expanding the norm yields
\begin{align*}
\|\theta^{\otimes m} - \phi^{\otimes m}\|^2 &= \|\theta\|^{2m} - 2 \langle \theta, \phi \rangle^m + \|\phi\|^{2m}\\
&=1-2(1+\gamma)^m+(1+2\gamma+\ep^2)^m\,,
\end{align*}
where $\gamma=\langle \theta, \phi - \theta\rangle$ is such that $|\gamma| \leq \eps$ by Cauchy-Schwarz.

By the binomial theorem, for all $x$ such that $|x| \leq 1$, there exists an $r_m$ such that
\begin{equation*}
(1 + x)^m = \sum_{k=0}^m \binom{m}{k} x^k = 1 + m x + r_m\,,
\end{equation*}
with $|r_m| \leq 2^m x^2$.
By assumption, $|\gamma| \leq \eps < 1$ and $2\gamma + \ep^2 \leq 3 \ep \leq 1$, so
\begin{align*}
\|\theta^{\otimes m} - \phi^{\otimes m}\|^2 & \leq 1 - 2 - 2 m \gamma + 2^{m+1} \ep^2 + 1 + 2 m \gamma + m \ep^2 + 2^m \cdot 9 \ep^2 \\
& \leq (m + 11 \cdot 2^m) \ep^2 \leq 12 \cdot 2^m \ep^2\,,
\end{align*}
proving the claim.

\end{proof}
\begin{appxlem}\label{lem:hermite_variance_bound}
For any symmetric tensors $\Delta_1$, \dots, $\Delta_{k}$, if $Y \sim \dist{P}_\zeta$, then
\begin{equation*}
\var\left(\sum_{m=1}^{k} \frac{\langle \Delta_m, H_m(Y) \rangle}{(\sqrt 3 \sigma)^{2m} m!}  \right)\leq e^{\|\zeta\|^2/2\sigma^2} \sum_{m=1}^{k} \frac{\|\Delta_m\|^2}{(\sqrt 3 \sigma)^{2m} m!}\,.
\end{equation*}
\end{appxlem}
\begin{proof}
Let $t(Y) = \sum_{m=1}^{k} \frac{\langle \Delta_m, H_m(Y) \rangle}{(\sqrt 3 \sigma)^{2m} m!}$.
To proceed, we bound the second moment $\E[t(Y)^2]$.
Denote by $\dist{P}_0$ the distribution $\cN(0, \sigma^2 I)$.
The Cauchy-Schwarz inequality implies that
\begin{align}
\E[t(Y)^2] 
& \leq \left[\E[t(Z)^4] (\chi^2(\dist{P}_\zeta, \dist{P}_0) + 1)\right]^{1/2}\,,  \quad \quad \text{where $Z \sim \dist{P}_0$.}\label{eqn:second_moment_bound}
\end{align}

Define the quantity
\begin{equation*}
\tilde{t}(Z) = \sum_{m=1}^{k} \frac{\langle \Delta_m, H_m(Y) \rangle}{(\sqrt 3)^m \sigma^{2m} m!}\,.
\end{equation*}
Standard facts about the Ornstein-Uhlenbeck semigroup (see~\cite[Proposition~11.37]{ODo14}) imply that
\begin{equation*}
t(Z) = \mathrm{U}_{1/\sqrt{3}} \tilde{t}(Z)\,,
\end{equation*}
where $U_{1/\sqrt{3}}$ is the operator defined by
\begin{equation*}
\mathrm{U}_{1/\sqrt{3}} \tilde t(z) = \E\big[ \tilde t\big(\frac{1}{\sqrt 3}(z + \sqrt{2} Z')\big)\big] \quad \quad \text{where $Z' \sim \cN(0, \sigma^2 I)$.}
\end{equation*}

By the Gaussian hypercontractivity inequality~\cite[Theorem 11.23]{ODo14},
\begin{equation*}
\E[t(Z)^4]^{1/2} \leq \E[\tilde{t}(Z)^2]\,.
\end{equation*}

For any multi-index $\alpha$ and $z \in \RR^L$, denote by $H_\alpha(z)$ the rescaled polynomial $\sigma^{|\alpha|} h_\alpha(\sigma^{-1} z_1, \dots, \sigma^{-1} z_L)$.
We note that the defining properties of the Hermite polynomials imply that if $Z \sim \cN(0, \sigma^2I)$, then $\E[H_\alpha(Z)] = 0$ and
\begin{equation*}
\E[H_\alpha(Z) H_\beta(Z)] = \left\{\begin{array}{ll}
\sigma^{2|\alpha|}\alpha! & \text{ if $\alpha = \beta$,} \\
0 & \text{ otherwise,}
\end{array}\right.
\end{equation*}
where $\alpha! = \alpha_1! \dots \alpha_m!$.

Since $\Delta_m$ is a symmetric tensor, the value $(\Delta_m)_{i_1 \dots i_m}$ depends only on the multi-set $\{i_1, \dots, i_m\}$, so for any multi-index $\alpha \in \N^d$ such that $|\alpha| = m$ corresponding to $\{i_1, \dots, i_m\}$, define
\begin{equation*}
\Delta_{\alpha} = (\Delta_m)_{i_1 \dots i_m}\,.
\end{equation*}
We obtain
\begin{equation*}
\langle \Delta_m, H_m(Z) \rangle = \sum_{\alpha: |\alpha| = m} \frac{m!}{\alpha!} \Delta_\alpha H_\alpha(Z)\,.
\end{equation*}

Therefore
\begin{align*}
\E[t(Z)^4]^{1/2} \leq \E [\tilde{t}(Y)^2] & = \sum_{m=1}^{k-1} \sum_{\alpha: |\alpha| = m} \frac{\Delta_\alpha^2}{(\sqrt{3})^{2m} \sigma^{4m} \alpha!^2}\E[H_\alpha(Z)^2] \\
&= \sum_{m=1}^{k-1} \sum_{\alpha: |\alpha| = m} \frac{\Delta_\alpha^2}{(\sqrt{3}\sigma)^{2m}\alpha!} \\
&= \sum_{m=1}^{k-1} \frac{\sum_{\alpha: |\alpha| = m} \frac{m!}{\alpha!}\Delta_\alpha^2}{(\sqrt{3}\sigma)^{2m}m!} \\
&= \sum_{m=1}^{k-1} \frac{\|\Delta_m\|^2}{(\sqrt 3 \sigma)^{2m} m!}\,.
\end{align*}

Denote by $\mathfrak{g}$ the density of a standard Gaussian random variable.
Then $\dist{P}_0$ has density $f_0(y) = \sigma^{-L} \mathfrak{g}(\sigma^{-1} y)$ and $\dist{P}_\zeta$ has density $f_\zeta(y) = \sigma^{-L} \E \mathfrak{g}(\sigma^{-1}(y - G \zeta)$.
Then
\begin{align*}
\chi^2(\dist{P}_\zeta, \dist{P}_0) + 1 & = \int \frac{f_\zeta(y)^2}{f_0(y)} \,\mathrm{d}y  \\
& \leq \int \sigma^{-L} \E \frac{\mathfrak{g}(\sigma^{-1}(y - G\zeta))^2}{\mathfrak{g}(\sigma^{-1}y)} \,\mathrm{d}y \\
& = \E_g \frac{1}{\sqrt{2 \pi \sigma^{2L}}} \int \exp\left(\frac{2y^\top G \zeta}{\sigma^2} - \frac{\|\zeta\|^2}{\sigma^2}\right) e^{-\|y\|^2/2\sigma^2} \, \mathrm{d}y \\
& = e^{\|\zeta\|^2/\sigma^2}\,,
\end{align*}
where in the second line we have applied Jensen's inequality.

Combining the above two bounds with~\eqref{eqn:second_moment_bound} yields that under $Y \sim \dist{P}_\zeta$,
\begin{equation*}
\var(t(Y)) \leq \E[t(Y)^2] \leq e^{\|\zeta\|^2/2\sigma^2}\sum_{m=1}^{k-1} \frac{\|\Delta_m\|^2}{(\sqrt 3 \sigma)^{2m} m!}\,,
\end{equation*}
as claimed.

\end{proof}
\begin{appxlem}\label{lem:derivative_bound}
Fix $\theta \in \RR^L$.
For any fixed $y \in \RR^L$, let $g(\phi) = \log \frac{f_\theta(y)}{f_\phi(y)} = \log \frac{\E \exp(-\frac{1}{2 \sigma^2} \|y - G \theta\|^2)}{\E \exp(-\frac{1}{2 \sigma^2} \|y - G \phi\|^2)}$.
Denote by $H(\phi)$ the Hessian of $g$ at $\phi$.
Then
\begin{equation*}
\|H(\phi)\|_{\mathrm{op}} \leq \sigma^{-2} + 2 \sigma^{-4}\|y\|^2
\end{equation*}
and
\begin{equation*}
\|H(\phi) - H(\eta)\|_{\mathrm{op}} \leq 6 \sigma^{-6} \|y\|^3\|\phi - \eta\|\,.
\end{equation*}
\end{appxlem}
\begin{proof}
Write $T^{(n)}_\zeta g$ for the $n$th derivative tensor of $g$ at $\zeta$: this is a symmetric tensor whose $(i_1, \dots, i_n)$ entry is $\frac{\partial^n g}{\partial \zeta_{i_1} \dots \partial \zeta_{i_n}}(\zeta)$.
Note that $T^{(2)}_\zeta g = H(\zeta)$.
Write $h(\zeta) = \E \exp(\frac{1}{\sigma^2} y^\top G \zeta)$.
The chain rule implies
\begin{equation*}
T^{(2)}_\zeta g = \frac{1}{\sigma^2}I - \frac{T^{(2)}_\zeta h}{h(\zeta)} + \left(\frac{T^{(1)}_\zeta h}{h(\zeta)}\right)^{\otimes 2}\,
\end{equation*}
and
\begin{equation*}
T^{(3)}_\zeta g = - \frac{T^{(3)}_\zeta h}{h(\zeta)} + 3 \,\mathrm{sym}\left(\frac{T^{(2)}_\zeta h}{h(\zeta)} \otimes \frac{T^{(1)}_\zeta h}{h(\zeta)}\right) - 2 \left(\frac{T^{(1)}_\zeta h}{h(\zeta)}\right)^{\otimes 3}\,,
\end{equation*}
where $\mathrm{sym}$ is the symmetrization operator which acts on order-3 tensors by averaging over all permutations of the indices:
\begin{equation*}
\mathrm{sym}(A)_{i_1 i_2 i_3} = \frac 1 6 \sum_{\pi \in \cS_3} A_{i_{\pi(1)}i_{\pi(2)}i_{\pi(3)}}\,.
\end{equation*}

By the Cauchy-Schwarz inequality,
\begin{align*}
\left|\left\langle \frac{T^{(n)}_\zeta h}{h(\zeta)}, u_1 \otimes \dots \otimes u_n\right\rangle\right| & = \left| \sigma^{-2n} \frac{\E [\prod_{i=1}^n (y^\top G u_i) \exp(\frac{1}{\sigma^2} y^\top G \zeta)]}{\E \exp(\frac{1}{\sigma^2} y^\top G \zeta)}\right| \\
& \leq \sigma^{-2n} \|y\|^n \prod_{i=1}^n \|u_i\|\,.
\end{align*}
This inequality implies
\begin{align*}
|\langle T^{(2)}_\zeta g, u^{\otimes 2} \rangle|& \leq \sigma^{-2}\|u\|^2 + 2 \sigma^{-4}\|y\|^2 \|u\|^2\,, \\
|\langle T^{(3)}_\zeta g, u_1 \otimes u_2 \otimes u_3\rangle|& \leq 6 \sigma^{-6} \|y\|^3 \|u_1\|\|u_2\|\|u_3\|\,.
\end{align*}
The first claim immediately follows.
For the second, we obtain
\begin{align*}
\|H(\phi) - H(\eta)\|_\text{op} & = \sup_{u \in \RR^d: \|u\| = 1} |u^\top H(\phi) u - u^\top H(\eta) u| \\
& = \sup_{u \in \RR^d: \|u\| = 1} |\langle T^{(2)}_\phi g, u \otimes u \rangle - \langle T^{(2)}_\eta g, u \otimes u \rangle| \\
& = \sup_{u \in \RR^d: \|u\| = 1} \left|\int_{0}^1 T^{(3)}_{\eta + \lambda(\phi - \eta)} g(u, u, \phi - \eta)\,\mathrm{d}\lambda\right| \\
& \leq 6 \sigma^{-6} \|y\|^3\|\phi - \eta\|\,.
\end{align*}
\end{proof}

\begin{appxlem}\label{lem:d_taylor}
There exists a positive constant $C_L$ depending on $L$ such that, for any $\theta \in \RR^L$ satisfying $\|\theta\| \leq 1$,
\begin{equation*}
|D(\theta \| \phi)  - \frac 1 2 (\phi - \theta)^\top H(\theta) (\phi - \theta)| \leq C_L \frac{\|\phi - \theta\|^3}{\sigma^3}\,.
\end{equation*}
\end{appxlem}
\begin{proof}
Denote by $T^{(3)}_\zeta$ the third derivative tensor of the function $\phi \mapsto D(\theta \, \| \, \phi)$ evaluated at $\phi = \zeta$.
By Taylor's theorem, there exists an $\eta$ on the segment between $\theta$ and $\phi$ such that
\begin{equation*}
D(\theta \| \phi) = \frac 1 2 (\phi - \theta)^\top H(\theta) (\phi - \theta) + \frac{1}{6} \langle T^{(3)}_\eta, (\phi - \theta)^{\otimes 3}\rangle\,.
\end{equation*}
It remains to bound the last term.
Note that
\begin{equation*}
D(\theta \| \phi) = \E\left[\frac{f_\theta(Y)}{f_\phi(Y)}\right]\,,
\end{equation*}
where $Y \sim \dist{P}_\theta$.
Therefore, by Lemma~\ref{lem:derivative_bound}, for any $u \in \RR^d$,
\begin{equation*}
|\langle T^{(3)}_\eta, u^{\otimes 3} \rangle| \leq 6 \E \sigma^{-6} \|Y\|^3 \|u\|^3 \leq C_L \sigma^{-3} \|u\|^3\,,
\end{equation*}
for some constant $C_L$ depending on $L$.
The claim follows.
\end{proof}

\begin{appxlem}\label{lem:second_moment_bound}
For all $\phi, \theta \in \RR^L$,
\begin{equation*}
\|\E[(G \theta)^{\otimes 2}] - \E[(G \phi)^{\otimes 2}]\|^2 \geq \frac{1}{L} (\|\theta\|^2 - \|\phi\|^2)^2\,.
\end{equation*}
Moreover, if $\|\phi\| \geq 3 \|\theta\|$, then
\begin{equation*}
\|\E[(G \theta)^{\otimes 2}] - \E[(G \phi)^{\otimes 2}]\|^2 \geq \frac{1}{4 L} \|\theta - \phi\|^4\,.
\end{equation*}
\end{appxlem}
\begin{proof}
Write $\Delta_2 = \E[(G \theta)^{\otimes 2}] - \E[(G \phi)^{\otimes 2}]$.
By the Cauchy-Schwarz inequality,
\begin{align*}
\big|\|\theta\|^2 - \|\phi\|^2\big|&=\big|\E\sum_{i=1}^L \big[( G \theta)_i^2-(G \phi)_i^2\big]\big|\\
&\leq\sqrt L\Big(\sum_{i=1}^L \big(\E [(G \theta)_i^2]-\E [(G \phi)_i^2]\big)^2\Big)^{1/2} \\
&\leq\sqrt L\Big(\sum_{i,j=1}^L \big(\E[( G \theta)_i( G \theta)_j]-\E[( G \phi)_i(G \phi)_j]\big)^2\Big)^{1/2}\\
& = \sqrt L \|\Delta_2\|\,.
\end{align*}
If $\|\phi\| \geq 3 \|\theta\|$, then
\begin{equation*}
\|\phi\|^2 - \|\theta\|^2 \geq \frac 1 2 (\|\phi\| + \|\theta\|)^2 \geq \frac 1 2 \|\phi - \theta\|^2\,.
\end{equation*}
We obtain
\begin{equation*}
\|\Delta_2\|^2 \geq \frac{1}{4L} \|\phi - \theta\|^4\,.
\end{equation*}
\end{proof}

\begin{appxlem}\label{lem:far_quadratic}
Let $\sigma \geq \|\theta\|$.
For all $\phi, \theta \in \RR^L$, if $\sigma \geq 1$, and $\metric(\theta, \phi) \geq 3 \sqrt 2 \sigma$, then $D(\theta \| \phi) \geq C_L \sigma^2 \metric(\theta, \phi)^2$, for some constant $C_L$ depending on $L$.
\end{appxlem}
\begin{proof}
We assume without loss of generality that $\metric(\theta, \phi) = \|\theta - \phi\|$.
Let $\varphi = \phi - \E[G\phi]$ and $\vartheta = \theta - \E[G \theta]$, and note that $\|\theta - \phi\|^2 = \|\E[G \theta] - \E[G \phi]\|^2 + \|\vartheta - \varphi\|^2$.
If $\|\E[G \theta] - \E[G \phi]\|^2 \geq \|\vartheta - \varphi\|^2$, then Lemma~\ref{lem:first_moment} implies
\begin{equation*}
D(\theta \| \phi) \geq \frac{1}{2 \sigma^2} \|\E[G \theta] - \E[G \phi]\|^2 \geq \frac{1}{4 \sigma^2} \|\theta - \phi\|^2\,,
\end{equation*}
which implies the claim.

On the other hand, if $\|\E[G \theta] - \E[G \phi]\|^2 < \|\vartheta - \varphi\|^2$, then the assumption that $\|\theta - \phi\| \geq 3 \sqrt 2 \sigma$ implies that $\|\vartheta - \varphi\|^2 \geq 9 \sigma^2 \geq 9 \|\theta\|^2$.
By Lemma~\ref{lem:second_moment_bound}, we obtain $\|\E[(G \vartheta)^{\otimes 2}] - \E[(G \varphi)^{\otimes 2}]\|^2 \geq \frac{1}{4L} \|\vartheta - \varphi\|^4 \geq \frac{9\sigma^2}{4 L} \|\vartheta - \varphi\|^2$.
Applying Theorem~\ref{thm:tight_KL_bound} implies
\begin{equation*}
D(\theta \| \phi) \geq C \sigma^{-4} \|\E[(G \vartheta)^{\otimes 2}] - \E[(G \varphi)^{\otimes 2}]\|^2 \geq C_L \sigma^{-2} \|\theta - \phi\|^2\,.
\end{equation*}
\end{proof}
\bibliographystyle{alpha}
\bibliography{mra_moment_matching}
\end{document}